\documentclass[11 pt]{article}
\usepackage{amsfonts,amsmath,color,amsthm,amssymb, url, enumerate, bbm}
\usepackage[margin=1in]{geometry}

\usepackage{tikz, etoolbox}
\usetikzlibrary{shapes}
\usetikzlibrary{arrows}
\usetikzlibrary{calc}

\newcommand{\2}{\sqrt{2}}

\newcommand{\f}{\frac}

\newcommand{\n}{\f{1}{n}}

\newcommand{\N}{\mathbb{N}}
\newcommand{\R}{\mathbb{R}}

\newcommand{\bi}{\mathbbm{1}}

\definecolor{purple}{RGB}{128,0,128}

\def\qed{\vbox{\hrule\hbox{\vrule\kern3pt\vbox{\kern6pt}\kern3pt\vrule}\hrule}}

\def\n{\noindent}
\newtheorem{thm}{Theorem}[section]
\newtheorem{lm}[thm]{Lemma}

\newtheorem{prop}[thm]{Proposition}

\newtheorem{cm}[thm]{Claim}
\newtheorem{conj}[thm]{Conjecture}

\newtheorem*{thm*}{Theorem}
\newtheorem*{cor*}{Corollary}
\newtheorem*{lm*}{Lemma}
\newtheorem*{cm*}{Claim}
\newtheorem*{prop*}{Proposition}

\theoremstyle{definition}

\newtheoremstyle%
 {Aside}%
 {}{}%
 {\color{purple}\itshape}
 {}%
 {\color{purple}\bfseries}%
 {\color{purple}.}%
 { }{}
\theoremstyle{Aside}

\title{The Circlet Inequalities: A New, Circulant-Based Facet-Defining Inequality for the TSP}
\author{Samuel C. Gutekunst and David P. Williamson}
\date{}

\begin{document}
\maketitle

\begin{abstract}
Facet-defining inequalities of the symmetric Traveling Salesman Problem (TSP) polytope play a prominent role in both polyhedral TSP research and  state-of-the-art TSP solvers.   In this paper, we introduce a new class of facet-defining inequalities, the \emph{circlet inequalities}.  These inequalities were first conjectured in Gutekunst and Williamson \cite{Gut19b} when studying Circulant TSP, and they provide a bridge between polyhedral TSP research and number-theoretic investigations of Hamiltonian cycles stemming from a conjecture due to Marco Buratti in 2017.   The circlet inequalities exhibit circulant symmetry by placing the same weight on all edges of a given length; our main proof exploits this symmetry to prove the validity of the circlet inequalities.  We then show that the circlet inequalities are facet-defining and compute their strength following  Goemans \cite{Goe95b}; they achieve the same worst-case strength as the similarly circulant crown inequalities of Naddef and Rinaldi \cite{Nad92}, but are generally  stronger.  
\end{abstract}

\section{Introduction and Outline}

The symmetric Traveling Salesman Problem (TSP) is a fundamental problem in combinatorial optimization, combinatorics, and theoretical computer science.  An instance consists of a set $[n]:=\{1, 2, 3, ..., n\}$ of $n$ cities and, for each pair of distinct cities $i, j \in [n]$, an associated cost or distance $c_{ij}=c_{ji}\geq 0$ reflecting the cost or distance of traveling from city $i$ to city $j$. The TSP is then to find a minimum-cost tour visiting each city exactly once.  Treating the cities as vertices of the complete, undirected graph $K_n,$ and treating an edge $\{i, j\}$ of $K_n$ as having cost $c_{ij}$, the TSP is equivalently to find a minimum-cost Hamiltonian cycle on $K_n.$

Let $STSP(n)$ denote the \emph{symmetric Traveling Salesman Problem polytope} on $n$ cities: the convex hull of the edge-incidence vectors of Hamiltonian cycles on $K_n$.  That is, $$STSP(n)=\text{conv}\{\chi_H : H \text{ is a Hamiltonian cycle on } K_n\}\subset \R^{|E|}.$$ A substantial theme of TSP research has been polyhedral, with a particular emphasis on \emph{facet-defining inequalities}.  These inequalities have played a fundamental role in developing TSP algorithms: despite the TSP being a fundamental NP-hard problem\footnote{The TSP with general symmetric edge costs is even hard to approximate within any constant factor $\alpha$: an algorithm that could find a Hamiltonian cycle in polynomial time, with that cycle guaranteed to cost no more than $\alpha$ times the (unknown) cost of the optimal solution would imply P=NP (for any constant alpha; see, e.g., Theorem 2.9 in Williamson and Shmoys \cite{DDBook}).  Even with more restrictive assumptions, such as that the edge costs are {\bf metric} (i.e. $c_{ij}\leq c_{ik}+c_{kj}$ for all distinct $i, j, k\in [n]$), it is known to be NP-hard to approximate TSP solutions in polynomial time to within any constant factor  $\alpha<\f{123}{122}$ (see Karpinski,  Lampis,  and Schmied \cite{Karp15}).}, the state-of-the-art TSP solver Concorde \cite{App06} has been able  to successfully solve structured instances with nearly 100,000 vertices! (See, e.g., Applegate, Bixby, Chvatal, and Cook \cite{App09}, which certifies the optimality of pla85900 from TSPLIB \cite{re91}.)

Formally, an inequality  $a x \geq a_0$  is \emph{valid} for $STSP(n)$ if every $x\in STSP(n)$ satisfies $a x \geq a_0.$    \emph{Facet-defining inequalities} are valid inequalities that define a \emph{facet} of $STSP(n)$ (i.e. letting $S=\{x\in STSP(n): ax = a_0\}$, the inequality $ax\geq a_0$ is facet defining when $dim(S)=dim(STSP(n))-1$, so that the inequality induces a facet of $STSP(n)$).   A nice survey of the prominent role of facet-defining inequalities in TSP research and computation can be found in Chapter 5 of  Applegate, Bixby, Chvatal, and Cook \cite{App06b}.  See also Gr{\"o}tschel and Padberg \cite{Gro86} and Naddef \cite{Nad06}.  Goemans \cite{Goe95b} surveys many facet defining inequalities for the TSP and provides a way to evaluate the strength of such inequalities.    Well-known inequalities for the TSP include the clique-tree inequalities (Gr{\"o}tschel and Pulleyblank \cite{Gro86b}), the comb inequalities (Chv{\'a}tal \cite{Chv73}, Gr{\"o}tschel and Padberg \cite{Gro79}), the crown inequalities (Naddef and Rinaldi \cite{Nad92}), the path inequalities (Cornu{\'e}jols,  Fonlupt,  and Naddef \cite{Cor85}), the path-tree inequalities (Naddef and Rinaldi \cite{Nad91}), the binested inequalities (Naddef \cite{Nad92b}), and the rank inequalities (Gr{\"o}tschel \cite{Gro80}).

The main result of this paper is a new facet-defining inequality arising from the \emph{Circulant TSP}, a special case of the TSP for which relatively little is known.  Circulant TSP instances are those whose edge costs can be described by a symmetric, \emph{circulant matrix}, a condition which imposes substantial symmetry: the cost of edge $\{i, j\}$ only depends on $(i-j)$ mod $n$.  The cost of an edge, then, depends only on its \emph{length}  $$\ell_{i, j}=\min\{|i-j|, n-|i-j|\}.$$   In analogue to a Circulant TSP instance, we define a \emph{circulant} inequality for the TSP as a valid inequality $a x \geq a_0$ where the coefficients of $a$ are circulant (i.e. $a_{i, j}=a_{i', j'}$ whenever $\ell_{i, j} = \ell_{i', j'}$).  Of the well-known facet-defining inequalities for $STSP(n),$ the crown Inequalities of Naddef and Rinaldi \cite{Nad92} are circulant.  

For a vector $x\in \R^{|E|}$, let $$t_i = \sum_{\{s, t\}\in E: \ell_{s, t} = i} x_{s, t}$$ denote the total weight of edges of length $i$.  A circulant inequality $a x\geq a_0$ can be rewritten as \begin{equation} \label{eq:circform} \sum_{i=1}^d c_i t_i \geq a_0, \end{equation} where $d = \lfloor\f{n}{2}\rfloor$ and $c_i = a_{0, i}$ is the cost of any edge of length $i$.

Notice that any valid inequality of the form  Equation (\ref{eq:circform}) expresses some requirement about edge lengths in a valid tour.  For example, suppose that $n$ is divisible by 4.  Then $t_{4} \leq n-4$ is a simple inequality stating that you cannot use more than $n-4$ edges of length 4:  edges of length 4 decompose the graph of $K_n$ into four distinct cycles, and a valid Hamiltonian cycle must use at most $n/4-1$ edges from any of these cycles.  Circulant inequalities thus offer a way to bridge together polyhedral investigations of the TSP in combinatorial optimization with related questions in number theory.   For example, the following conjecture dates to Marco Buratti in 2007  and conjectures conditions for a Hamiltonian path using prescribed edge lengths (see Buratti and Merola \cite{Bur13} for an initial statement, Horak and Rosa \cite{Hor09} for a generalization, and Pasotti and Pellegrini \cite{Pas14} for a rephrasal). 

\begin{conj}[Buratti]
	Let $L$ be a multiset of size $n-1$ consisting of edge lengths in $1, 2, ...., \lfloor \f{n}{2} \rfloor.$  There exists a Hamiltonian path in $K_n$ using edge lengths $L$ if and only if: for every $q$ that divides $n$, $$\#\{e\in L: q|e\} \leq n-q.$$  
\end{conj} 
\noindent Here $\{e\in L: q|e\}$ is taken as a multiset consisting of all edge lengths in $L$ that are a multiple of $q$.  In the case where $n=8$, for example, this condition says that $L$ must contain at most $8-2=6$ edges of even length.  Any analogue of Buratti's condition for Hamiltonian paths would give rise to a circulant inequality for the TSP.  Costa, Morini, Pasotti, and Pellegrini \cite{cost80}, e.g., explicitly leave as open finding necessary and sufficient conditions for a graph to have a Hamiltonian cycle using prescribed edge lengths.

The circulant facet-defining inequality we propose (the \emph{circlet inequality}) here takes the following form:  Suppose that  $n$ is divisible by 4 and let $d=\f{n}{2}.$ Then for any $x\in\R^{|E|}$ that is a feasible TSP input, \begin{equation}\label{eq:main}\sum_{i=1}^d c_i t_i \geq n-2, \hspace{10mm} c_i = \begin{cases} i, & i\text{ odd} \\ d-i, & i \text{ even}.\end{cases}\end{equation}
Note that this inequality has an ``alternating'' structure based on parity: for $i$ odd, the coefficient of $t_i$ grows with $i$, while for $i$ even, the coefficient of $t_i$ decreases with $i.$  When $n=12,$ e.g., the inequality is $$t_1+4t_2 +3t_3 + 2t_4 + 5 t_5 + 0t_6 \geq 10.$$

Note that this inequality is valid for any arbitrary labeling of the vertices.  Hence, for any TSP instance, there are on the order of $n!$ possible versions that can be applied.  Consider a point $x\in \R^{|E|}$ that may or may not be in $STSP(n)$.  One can arbitrarily relabel the vertices with any permutation of $1, ..., n$, determine the `edge lengths' and $t_i$ values based on that new labeling, and apply Inequality (\ref{eq:main}).  If $x$ is not feasible for Inequality (\ref{eq:main}) under that relabeling, then $x$ is not in $STSP(n)$.  

In addition to the number theoretic connections and interpretations of the circlet inequalities, we highlight three additional properties of it.  First, Inequality (\ref{eq:main}) was first conjectured in Gutekunst and Williamson \cite{Gut19b}, and motivated by Circulant TSP.  Gutekunst and Williamson \cite{Gut19b} specifically characterized the \emph{integrality gap} of the prototypical LP relaxation of the TSP, the subtour LP.  Specifically, Gutekunst and Williamson showed that the worst case ratio of the subtour LP relative to the TSP is exactly 2 on Circulant TSP instances.  Gutekunst and Williamson consider several avenues to improving the integrality gap on Circulant TSP instances.  They noted that many facet-defining inequalities eliminate the specific solution used to show an integrality gap of 2, including the ladder, chain, and crown inequalities (see Boyd and Cunningham \cite{Boyd91}, Padberg and Hong \cite{Pad80}, and Naddef and Rinaldi \cite{Nad92}). However, none of these inequalities were robust to a small modification of the specific solution used.  Gutekunst and Williamson conjectured  Inequality (\ref{eq:main}) and noted that -- if valid -- adding Inequality (\ref{eq:main}) would robustly eliminate the specific solution; see Section \ref{sec:back} for more details.  Thus this paper resolves the Circulant TSP conjecture of Gutekunst and Williamson \cite{Gut19b}.

Second, Inequality  (\ref{eq:main}) can itself be considered as defining a circulant, non-metric TSP instance. One places a cost $c_i$ on each edge of length $i$ and verifies if the minimum cost solution to that TSP instance costs at least $n-2$; see Figure \ref{fig:sym} for an example of the symmetry of such an instance.  Indeed, when investigating this TSP inequality, we first experimentally tried to verify its validity using Concorde \cite{App06}.  Our instances presents potential computational novelty: despite having solved instances with nearly 100,000 vertices, Concorde struggled to verify the circlet inequality on even tiny instances\footnote{Bill Cook, one of the authors of Concorde, generously ran our instances on Concorde and it took over 40 hours to verify the circlet inequality when $n=32$. He noted that he sometimes found difficult small instances for Concorde to solve, ``but [that] 32 nodes might be a record."}. 

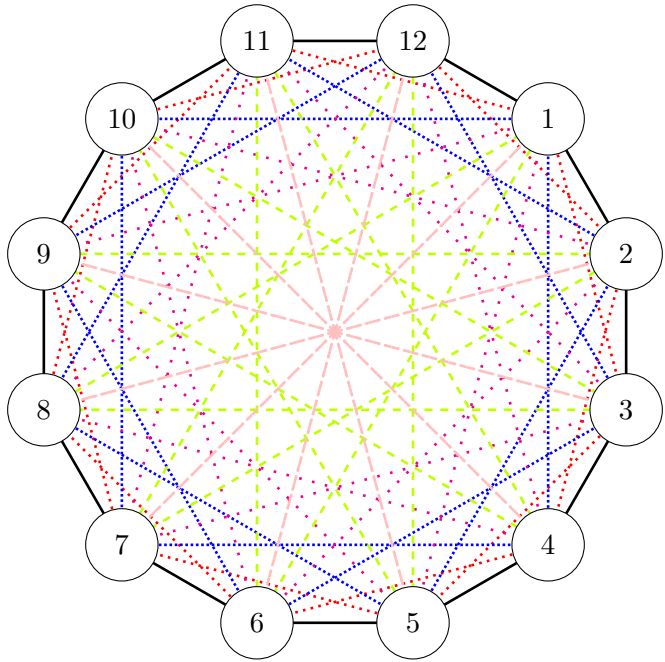
\begin{figure}[h!]
	\begin{center}
\begin{tikzpicture}
\tikzset{vertex/.style = {shape=circle,draw,minimum size=2.5em}}
\tikzset{edge/.style = {->,> = latex'}}
\node[draw=none,minimum size=8cm,regular polygon,regular polygon sides=12] (a) {};

\foreach [evaluate ={\j=int(mod(\i, 12)+1)}] \i in {1,2, 3, ..., 12}
\draw[line width=1pt] (a.corner \i) -- (a.corner \j);

\foreach [evaluate ={\j=int(mod(\i+1, 12)+1)}] \i in {1,2, 3, ..., 12}
\draw[dotted, line width=1pt, red] (a.corner \i) -- (a.corner \j);

\foreach [evaluate ={\j=int(mod(\i+2, 12)+1)}] \i in {1,2, 3, ..., 12}
\draw[densely dotted, line width=1pt, blue] (a.corner \i) -- (a.corner \j);

\foreach [evaluate ={\j=int(mod(\i+3, 12)+1)}] \i in {1,2, 3, ..., 12}
\draw[loosely dotted, line width=1pt, magenta] (a.corner \i) -- (a.corner \j);

\foreach [evaluate ={\j=int(mod(\i+4, 12)+1)}] \i in {1,2, 3, ..., 12}
\draw[dashed, line width=1pt, lime] (a.corner \i) -- (a.corner \j);

\foreach [evaluate ={\j=int(mod(\i+5, 12)+1)}] \i in {1,2, 3, ..., 12}
\draw[dashdotted, line width=1pt, pink] (a.corner \i) -- (a.corner \j);

\foreach \n [count=\nu from 1, remember=\n as \lastn, evaluate={\nu+\lastn}] in {12, 11, ..., 1} 
\node[vertex, fill=white]at(a.corner \n){$\nu$};

\end{tikzpicture}

	\end{center}
	\caption{Circulant symmetry.  Edges of a fixed length are indistinguishable.  E.g. all edges of the form $\{v, v+1\}$ (where $v+1$ is taken mod $n$) have the same appearance. }\label{fig:sym}\end{figure}

Finally, we hope  that this paper motivates a search for other circulant facet-defining inequalities.  Such inequalities are intimately connected to number theory, and may provide a new approach to the TSP: decades of research have still not resolved many questions about $STSP(n)$, and one might wonder if a more number-theoretic take -- an understanding of the combinations of edge lengths that can constitute a Hamiltonian cycle -- might provide new insights.  We specifically consider the projection of $STSP(n)$ to the variables $t_1, ..., t_d.$  Let $$EL(n):= \text{conv}\{(t_1, ..., t_d): x\in STSP(n), t_i = \sum_{\{s, t\}\in E: \ell_{s, t} = i} x_{s, t}\}$$ denote the edge-length TSP polytope.  One might ask if it is possible to characterize $EL(n)$, and if valid inequalities for $EL(n)$ are useful in solving  TSP instances.   

\subsection{Outline}

We begin by providing brief background on Circulant TSP and context for our results  in Section \ref{sec:back}.  Then we are able to present the two main theorems in this paper: Theorem \ref{thm:main}, which proves that Inequality (\ref{eq:main}) is valid, and Theorem \ref{thm:tight}, which proves that it is facet-defining.

Sections \ref{sec:mainpf} and \ref{sec:lem} present the proof of Theorem \ref{thm:main}, which uses two  lemmas which provide contrasting conditions on potential counter-examples to Inequality (\ref{eq:main}).  Recall that this equation is of the form $\sum_{i=1}^d c_i t_i \geq n-2$ where $c_i, t_i \geq 0$ and $\sum_{i=1}^d t_i = n.$  On one hand,  $c_1 = 1, c_d =0$ and $c_i\geq 2$ otherwise.  Hence, any possible  counter-example to Inequality (\ref{eq:main}) requires $t_1$ and $t_d$ to be large: for $\sum_{i=1}^d c_i t_i <n-2$, the cost of an average edge must be strictly less than 1, and all other edges of length $i\not\in\{1, d\}$ cost at least twice that.  We formalize this observation in our first lemma,   Lemma \ref{lm1}.

On the other hand,  our more  technical   lemma, Lemma \ref{lm2}, argues that a minimal counterexample cannot be ``dense'' in edges of length 1 and $d$: any ``window'' of four vertices $v, v+1, v+d, v+d+1$ can include at most 2 such edges.  By arguing that the conditions of Lemmas \ref{lm1} and \ref{lm2} are mutually incompatible, we can quickly deduce our main result.  We do so in Section \ref{sec:mainpf}, where we provide the proof up to Lemma \ref{lm2}. Then, in Section \ref{sec:lem}, we prove Lemma \ref{lm2}.  Lemma \ref{lm2} is considerably more involved than Lemma \ref{lm1}, and to prove it, we carefully consider what happens when we contract any window using 3 edges of length 1 and $d$.  Doing so involves careful bookkeeping on how edge costs change under contraction, in addition to combinatorial observations about the different types of edges a tour can contain.

In Section \ref{sec:facet}, we turn to our second main theorem.  By presenting a relatively small set of tours and exploiting symmetry, we can quickly show that Inequality (\ref{eq:main}) is tight.  Finally, in Section \ref{sec:Strength}, we turn towards analyzing the strength of Inequality (\ref{eq:main}).  We compute its strength following Goemans \cite{Goe95b}, and we show that the strength of Inequality (\ref{eq:main}) is $$\f{n^2-2n-4}{n^2-3n}\leq \f{11}{10}.$$  It is equal to $\f{11}{10}$ when $n=8.$   For comparison, we note that the bound of $\f{11}{10}$ is also attained when $n=8$ by the crown inequality; otherwise ours is marginally stronger than the crown inequality (Naddef and Rinaldi \cite{Nad92}).

\section{Background}\label{sec:back}

\subsection{Circulant TSP}

As noted above, Circulant TSP instances are those whose edge costs can be described by a symmetric, \emph{circulant matrix}.  Since the cost of edge $\{i, j\}$ only depends on $(i-j)$ mod $n$, the cost matrix is in terms of $ \lfloor \frac{n}{2}\rfloor$ parameters: \begin{equation}\label{eq:ScircMat} C=(c_{(j-i) \text{ mod } n})_{i, j=1}^n=\begin{pmatrix} 0 & c_1 & c_2 & c_3 & \cdots & c_{1} \\ c_{1} & 0 & c_1 & c_2 & \cdots & c_2 \\ c_{2} & c_{1} &0 & c_1 & \ddots & c_{3} \\ \vdots & \vdots & \vdots & \vdots & \ddots & \vdots \\ c_1 & c_2 & c_3 & c_4 & \cdots & 0\end{pmatrix},\end{equation} with $c_0=0$ and  $c_i=c_{n-i}$ for $i=1, ..., \lfloor \frac{n}{2}\rfloor.$    Importantly, in Circulant TSP there is not necessarily an assumption that the edge costs are metric.

Circulant TSP initially arose from questions of minimizing wallpaper waste in Garfinkel \cite{Gar77} and reconfigurable network design in Medova \cite{Med93}. One of the reasons that Circulant TSP has remained so compelling is that circulant instances seem to provide just enough structure to make an ambiguous set of instances: it is unclear whether or not a given combinatorial optimization problem should remain hard or become easy when restricted to circulant instances.  Some classic combinatorial optimization problems become easy when restricted to circulant instances. In the late 80's, Burkard and Sandholzer \cite{Burk91} showed that the decidability question for whether or not a symmetric circulant graph is Hamiltonian can be solved in polynomial time and showed that  bottleneck TSP is polynomial-time solvable on symmetric circulant graphs.  Bach, Luby, and Goldwasser (cited in Gilmore, Lawler, and Shmoys \cite{Gil85}) showed that one could find minimum-cost Hamiltonian paths in (not-necessarily-symmetric) circulant graphs in polynomial time.  In contrast, Codenotti, Gerace, and Vigna \cite{Code98} showed that Max Clique and Graph Coloring remain NP-hard when restricted to circulant graphs and do not admit constant-factor approximation algorithms unless P=NP.

Gutekunst and Williamson \cite{Gut19b} analyze the  prototypical LP relaxation of $STSP(n)$ on Circulant TSP instances.  This LP relaxation is the \emph{subtour elimination linear program} (also referred to as the Dantzig-Fulkerson-Johnson relaxation \cite{Dan54} and the Held-Karp bound \cite{Held70}, and which we will refer to as the {\bf subtour LP} and whose feasible region we will abbreviate as $SP(n)$).  The subtour LP has a variable $x_e$ associated to each edge: 
\begin{equation}\label{eq:stlp}
\begin{array}{l l l}
\min & \sum_{e\in E} c_e x_e & \\
\text{subject to} & \sum_{e\in \delta(v)} x_e = 2, & v=1, \ldots, n \\
& \sum_{e\in \delta(S)} x_e \geq 2, & S\subset V: S\neq \emptyset, S\neq V \\
&0\leq x_e \leq 1, & e=1, \ldots, n.
\end{array} \end{equation}
Given a Hamiltonian cycle $\mathcal{C}$, there is a  feasible solution to the subtour LP attained by setting $x_e=1$ for each $e\in \mathcal{C}$ and $x_e=0$ otherwise.  When edge costs are metric, Wolsey \cite{Wol80},  Cunningham \cite{Cun86}, and Shmoys and Williamson \cite{Shm90} show that solutions to this linear program are within a factor of $\f{3}{2}$ of the optimal, integer solution to the TSP. 
\begin{thm}[Wolsey \cite{Wol80},  Cunningham \cite{Cun86}, and Shmoys and Williamson \cite{Shm90}]
	The integrality gap of the subtour LP on metric TSP instances is at most $\f{3}{2}.$  That is, for any input to the TSP with metric edge costs, the ratio $$\f{\text{Cost of Optimal TSP solution}}{\text{Cost of Optimal LP Solution}} \leq \f{3}{2}.$$
\end{thm}
\noindent It is conjectured that the integrality gap of the subtour LP on metric TSP instances is at most $\f{4}{3},$ and one motivation for this conjecture stems from the definition of strength in Goemans \cite{Goe95b}.  The $\f{3}{2}$ bound, however, remains state of the art.  

Gutekunst and Williamson \cite{Gut19b}  show that \emph{integrality gap} of the subtour LP on circulant instances -- the worst case ratio of the subtour LP relative to the TSP -- is exactly 2.  Figure \ref{fig:intGap} describes the circulant (but non-metric) TSP instances and corresponding subtour LP solution used to show that the integrality gap on circulant instances is at least 2.  These instances have $n=2^{k+1}$ vertices.  Edges of length $d$ have cost 0, edges of length $1$ have cost 1, and all other edges have arbitrarily large costs.  Gutekunst and Williamson argue that the cheapest possible tour costs $2^{k+1}-2$.  In contrast, the solution shown in Figure \ref{fig:intGap} places weight $1/2$ on all edges of length 1 and weight $1$ on all edges of length $d$.  It is feasible for the subtour LP, and the cost of such a subtour LP solution is only $2^k.$

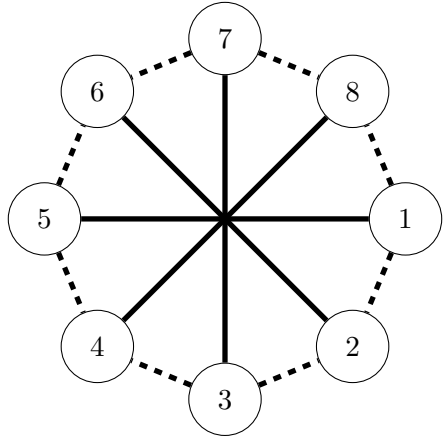
\begin{figure}[t]
	\centering
	
	\begin{tikzpicture}[scale=0.6]
	\tikzset{vertex/.style = {shape=circle,draw,minimum size=2.5em}}
	\tikzset{edge/.style = {->,> = latex'}}
	\tikzstyle{decision} = [diamond, draw, text badly centered, inner sep=3pt]
	\tikzstyle{sq} = [regular polygon,regular polygon sides=4, draw, text badly centered, inner sep=3pt]
	\node[vertex] (a) at  (4, 0) {$1$};
	\node[vertex] (b) at  (2.828, 2.828) {$8$};
	\node[vertex] (c) at  (0, 4) {$7$};
	\node[vertex] (d) at  (-2.828,2.828) {$6$};
	\node[vertex] (e) at  (-4, 0) {$5$};
	\node[vertex] (f) at  (-2.828, -2.828) {$4$};
	\node[vertex] (g) at  (0, -4) {$3$};
	\node[vertex] (h) at  (2.828, -2.828) {$2$};

	\draw[dashed, line width=2pt] (a) to (b);
	\draw[dashed, line width=2pt] (c) to (b);
	\draw[dashed, line width=2pt] (c) to (d);
	\draw[dashed, line width=2pt] (e) to (d);
	\draw[dashed, line width=2pt] (e) to (f);
	\draw[dashed, line width=2pt] (g) to (f);
	\draw[dashed, line width=2pt] (g) to (h);
	\draw[dashed, line width=2pt] (a) to (h);
	
	\draw[line width=2pt] (e) to (a);
	\draw[line width=2pt] (b) to (f);
	\draw[line width=2pt] (c) to (g);
	\draw[line width=2pt] (d) to (h);
	\end{tikzpicture}
	\caption{An example of a class of instances showing that the integrality gap of the subtour LP restricted to circulant instances is at least 2.  The dashed edges have weight $1/2$ and cost 1, while the full edges have weight 1 and cost 0.  All other edges have arbitrarily large cost} \label{fig:intGap}
\end{figure} 

Gutekunst and Williamson noted that many facet-defining inequalities eliminate the specific subtour LP solution indicated in Figure \ref{fig:intGap}, including the ladder, chain, and crown inequalities (see Boyd and Cunningham \cite{Boyd91}, Padberg and Hong \cite{Pad80}, and Naddef and Rinaldi \cite{Nad92}): any of these inequalities could be added to the subtour LP to potentially strengthen its integrality gap on Circulant TSP instances.  Indeed, the crown inequalities of Naddef and Rinaldi \cite{Nad92} are also motivated by the exact same subtour LP solution as as shown in Figure \ref{fig:intGap}!

However, Gutekunst and Williamson also noted that a cursory modification to these subtour LP weights -- marginally increasing the weight on length-1 edges and decreasing the weight on length-$d$ edges -- yields edge weights that 1) are feasible for the subtour LP, 2) obey the ladder, chain, and crown inequalities, and 3) still show that the integrality gap of the subtour LP is 2 on circulant instances.  

More specifically, consider solutions that place a weight of $\lambda$ on every edge of length 1, and a weight of $2-2\lambda$ on every edge of length $d$.  Such a solution is only in $STSP(n)$ if $$1-\f{2}{n} \leq \lambda \leq 1.$$  However, adding the crown inequalities, e.g., only imposes that $$\lambda \geq \f{1}{2}+\f{2}{3n}+\f{1}{3(n-6)}.$$  Gutekunst and Williamson conjectured  Inequality (\ref{eq:main}) as a way to eliminate this entire family of bad instances. Consider instances on $n=2^{k+1}$ vertices and potential solutions that place a weight of $\lambda$ on every edge of length 1, and a weight of $2-2\lambda$ on every edge of length $d$.  Then Inequality (\ref{eq:main}) directly implies $$n\lambda + 0(\f{n}{2})(2-2\lambda) \geq n-2.$$  That is, that $$\lambda \geq 1-\f{2}{n}.$$  Inequality (\ref{eq:main}) thus takes a canonically bad family of subtour LP solutions, and eliminates every single instance in that family that is outside $STSP(n)$.

\section{Theorem \ref{thm:main}} \label{sec:mainpf}

\subsection{Preliminaries: Notation and Lemmas}
We briefly recall our notation.  Let $n=2d$ where $d$ is even.  For a TSP instance on $[n]$, let $\ell_{i, j}$ denote the length of edge $\{i, j\},$ so that $$\ell_{i, j}=\min\{|i-j|, n-|i-j|\}.$$  Let $x_e$ denote the standard TSP variables associated to each edge $e\in E$ of the complete graph.  For a vector $x\in\R^{|E|}$ and for $1\leq i\leq d$, let $$t_i = \sum_{\{s, t\}\in E: \ell_{s, t} = i} x_{s, t}$$ denote the total weight of edges of length $i$.

\begin{thm}\label{thm:main}
For $n$ divisible by 4 and $1\leq i\leq d,$ let $$c_i = \begin{cases} i, & i\text{ odd} \\ d-i, & i \text{ even}.\end{cases}$$  Then for any $x\in\R^E$ that is a feasible TSP input, \begin{equation}\sum_{i=1}^d c_i t_i \geq n-2. \tag{\ref{eq:main}}\end{equation}  That is, the {\bf circlet inequalities} are valid.
\end{thm}
Figure \ref{fig:tight} shows some instances for which Inequality (\ref{eq:main}) is tight.

\begin{figure}[h!]
	\begin{center}

\begin{tikzpicture}[scale=0.5]
\tikzset{vertex/.style = {shape=circle,draw,minimum size=3.2em}}
\tikzset{edge/.style = {->,> = latex'}}
\tikzstyle{decision} = [diamond, draw, text badly centered, inner sep=3pt]
\tikzstyle{sq} = [regular polygon,regular polygon sides=4, draw, text badly centered, inner sep=3pt]
\node[vertex] (a) at  (0, 2) {$1$};
\node[vertex] (b) at  (4, 2) {$2$};
\node[vertex] (c) at  (8, 2) {$3$};
\node (d) at  (12, 2) {$\cdots$};
\node[vertex] (e) at  (16, 2) {$d-2$};
\node[vertex] (f) at  (20, 2) {$d-1$};
\node[vertex] (g) at  (24, 2) {$d$};

\node[vertex] (a1) at  (0, -2) {$d+1$};
\node[vertex] (b1) at  (4, -2) {$d+2$};
\node[vertex] (c1) at  (8, -2) {$d+3$};
\node (d1) at  (12, -2) {$\cdots$};
\node[vertex] (e1) at  (16, -2) {$n-2$};
\node[vertex] (f1) at  (20, -2) {$n-1$};
\node[vertex] (g1) at  (24, -2) {$n$};

\draw[line width=2pt]  (a) -- (b);
\draw[line width=2pt]  (b) -- (c);
\draw[line width=2pt]  (d) -- (c);
\draw[line width=2pt]  (d) -- (e);
\draw[line width=2pt]  (e) -- (f);
\draw[line width=2pt]  (g) -- (f);

\draw[line width=2pt]  (a1) -- (b1);
\draw[line width=2pt]  (b1) -- (c1);
\draw[line width=2pt]  (d1) -- (c1);
\draw[line width=2pt]  (d1) -- (e1);
\draw[line width=2pt]  (e1) -- (f1);
\draw[line width=2pt]  (g1) -- (f1);

\draw[line width=2pt]  (a) -- (a1);
\draw[line width=2pt]  (g) -- (g1);

\end{tikzpicture}

\vspace{1.5cm}

\begin{tikzpicture}[scale=0.5]
\tikzset{vertex/.style = {shape=circle,draw,minimum size=3.2em}}
\tikzset{edge/.style = {->,> = latex'}}
\tikzstyle{decision} = [diamond, draw, text badly centered, inner sep=3pt]
\tikzstyle{sq} = [regular polygon,regular polygon sides=4, draw, text badly centered, inner sep=3pt]
\node[vertex] (a) at  (0, 2) {$1$};
\node[vertex] (b) at  (4, 2) {$2$};
\node[vertex] (c) at  (8, 2) {$3$};
\node (d) at  (12, 2) {$\cdots$};
\node[vertex] (e) at  (16, 2) {$d-2$};
\node[vertex] (f) at  (20, 2) {$d-1$};
\node[vertex] (g) at  (24, 2) {$d$};

\node[vertex] (a1) at  (0, -2) {$d+1$};
\node[vertex] (b1) at  (4, -2) {$d+2$};
\node[vertex] (c1) at  (8, -2) {$d+3$};
\node (d1) at  (12, -2) {$\cdots$};
\node[vertex] (e1) at  (16, -2) {$n-2$};
\node[vertex] (f1) at  (20, -2) {$n-1$};
\node[vertex] (g1) at  (24, -2) {$n$};

\draw[line width=2pt]  (a) -- (b);
\draw[line width=2pt]  (d) -- (c);
\draw[line width=2pt]  (d) -- (e);
\draw[line width=2pt]  (g) -- (f);

\draw[line width=2pt]  (b1) -- (c1);
\draw[line width=2pt]  (e1) -- (f1);

\draw[line width=2pt]  (a) -- (a1);
\draw[line width=2pt]  (b) -- (b1);
\draw[line width=2pt]  (c) -- (c1);
\draw[line width=2pt]  (e) -- (e1);
\draw[line width=2pt]  (f) -- (f1);
\draw[line width=2pt]  (g) -- (g1);

\draw[line width=2pt]  (a1) to [out=-25, in=-155] (g1);

\end{tikzpicture}

\vspace{1.5cm}

\begin{tikzpicture}[scale=0.5]
\tikzset{vertex/.style = {shape=circle,draw,minimum size=3.2em}}
\tikzset{edge/.style = {->,> = latex'}}
\tikzstyle{decision} = [diamond, draw, text badly centered, inner sep=3pt]
\tikzstyle{sq} = [regular polygon,regular polygon sides=4, draw, text badly centered, inner sep=3pt]
\node[vertex] (a) at  (0, 2) {$1$};
\node[vertex] (b) at  (4, 2) {$2$};
\node[vertex] (c) at  (8, 2) {$3$};
\node (d) at  (12, 2) {$\cdots$};
\node[vertex] (e) at  (16, 2) {$d-2$};
\node[vertex] (f) at  (20, 2) {$d-1$};
\node[vertex] (g) at  (24, 2) {$d$};
\node (h) at  (27, 2) {};

\node (h1) at  (-3, -2) {};
\node[vertex] (a1) at  (0, -2) {$d+1$};
\node[vertex] (b1) at  (4, -2) {$d+2$};
\node[vertex] (c1) at  (8, -2) {$d+3$};
\node (d1) at  (12, -2) {$\cdots$};
\node[vertex] (e1) at  (16, -2) {$n-2$};
\node[vertex] (f1) at  (20, -2) {$n-1$};
\node[vertex] (g1) at  (24, -2) {$n$};

\draw[line width=2pt]  (a) -- (b);
\draw[line width=2pt]  (b) -- (c);
\draw[line width=2pt]  (d) -- (c);
\draw[line width=2pt]  (d) -- (e);
\draw[line width=2pt]  (e) -- (f);

\draw[line width=2pt]  (b1) -- (c1);
\draw[line width=2pt]  (d1) -- (c1);
\draw[line width=2pt]  (d1) -- (e1);
\draw[line width=2pt]  (e1) -- (f1);

\draw[line width=2pt]  (a) -- (a1);
\draw[line width=2pt]  (g) -- (g1);
\draw[line width=2pt]  (f) -- (f1);

\draw[line width=2pt]  (b1) to [out=-25, in=-155] (g1);

\draw[line width=2pt, dashed]  (g) -- (h);
\draw[line width=2pt, dashed]  (h1) -- (a1);
\end{tikzpicture}

	\end{center}
	\caption{Instances for which Inequality  (\ref{eq:main}) is tight.  Any present edges indicate a weight of 1 on the corresponding LP variable.  All other edges have weight 0.  The first figure shows a tour using two edges of length $d$ and $n-2$ edges of length 1.  The second tour using all $d$ edges of length $d$, $d-1$ edges of length $1$, and one edge of length $d-1$.  The third tour (which includes the edge $\{d, d+1\}$ indicated by a dashed line) uses 3 edges of length $d$,  $n-4$ edges of length 1, and one edge of length $d-2.$ }\label{fig:tight}\end{figure}
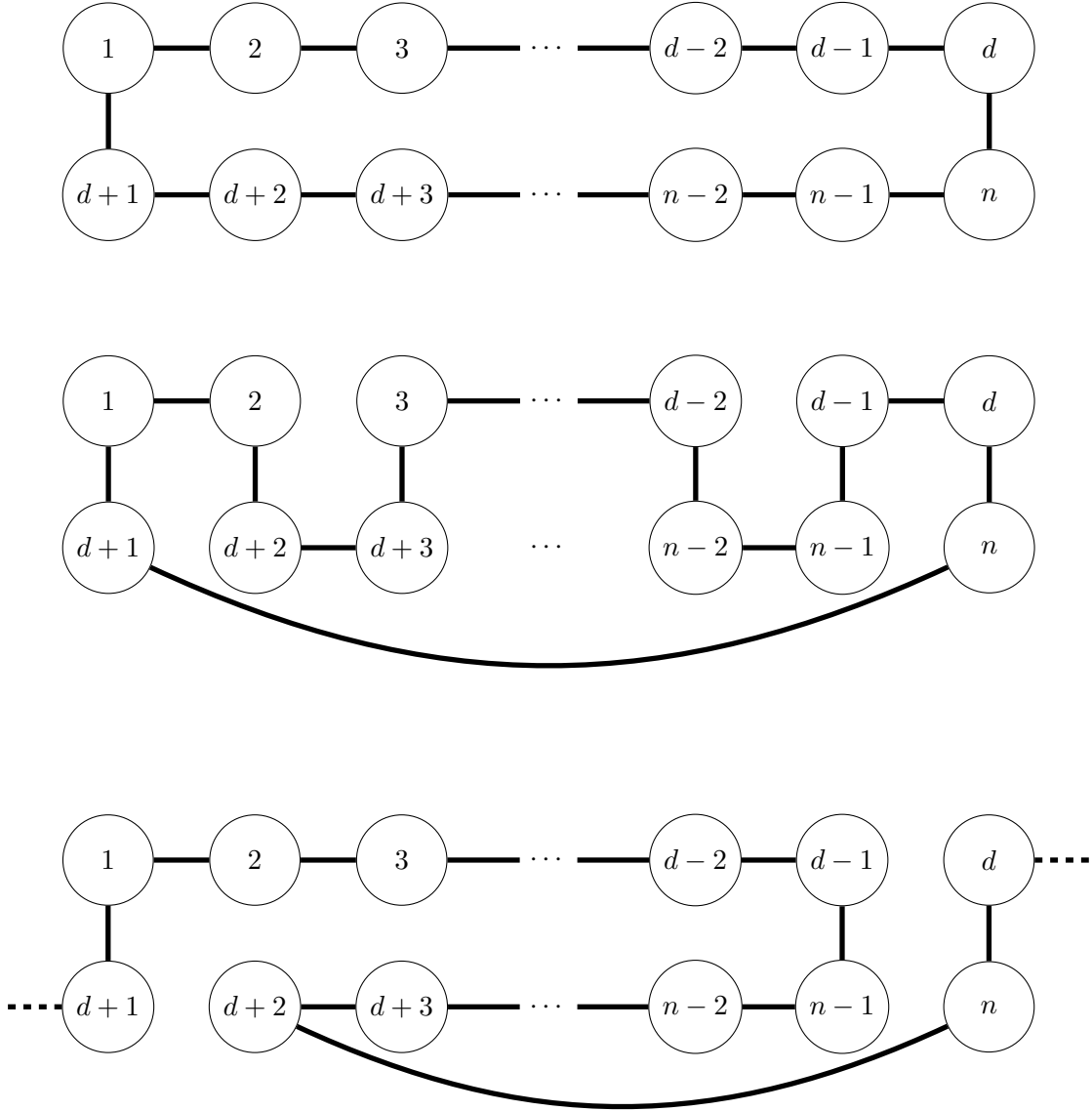

Our proof of Theorem \ref{thm:main} uses three ingredients.  First, Claim \ref{cm1} provides a base case.

\begin{cm}\label{cm1}
Inequality  (\ref{eq:main})  is valid for $n=4.$
\end{cm}

\begin{proof}
When $n=4$, the inequality becomes $$t_1+0t_2\geq 2.$$  Any Hamiltonian cycle must use four edges, the only possible edge lengths are 1 and 2, and there are only two distinct edges of length 2.  Hence the claim holds.
\hfill
\end{proof}

Our next lemma argues that any counterexample requires many cheap edges: edges of length 1 and $d$.

\begin{lm}\label{lm1}
Suppose that we have a valid TSP instance where $$\sum_{i=1}^d c_i t_i < n-2.$$ Then $$t_1+2t_d > n+2.$$
\end{lm}

\begin{proof}
In a counterexample,
\begin{align*}
n-2 &> \sum_{i=1}^d c_i t_i \\
&\geq t_1 + 0 t_d + 2\sum_{i=2}^{d-1} t_i \\
&= t_1 + 0 + 2 (n-t_1-t_d) \\
& = 2n -t_1 - 2t_d.
\end{align*}

Rearranging $$n-2 > 2n -t_1 - 2t_d$$  yields the desired inequality. \hfill
\end{proof}

Finally, we state our technical lemma.  Its proof is deferred to Section \ref{sec:lem}.  In the notation of the lemma -- and throughout this paper -- vertex labels are implicitly assumed to be taken mod $n$ (e.g. for a vertex $v\in [n],$ we write $v+k$ to denote $v+k\mod n$).

\begin{lm}\label{lm2}
Suppose that we have a valid TSP instance where $$\sum_{i=1}^d c_i t_i < n-2,$$ and consider an instance that is minimal with respect to $n$.  Then the counterexample cannot have any of the structures shown in Figure \ref{fig:bad}.  That is, for any $u\in [n],$ a minimal counterexample cannot contain the three edges $\{u+d, u\}, \{u, u+1\}, \{u+1, u+1+d\};$ a minimal counterexample cannot contain the three edges $\{u+1, u\}, \{u, u+d\}, \{u+d, u+1+d\};$ and a minimal counterexample cannot contain the three edges $ \{u, u+1\}, \{u+1, u+1+d\}, \{u+1+d, u+d\}.$
\end{lm}

\begin{figure}[h!]
\begin{center}

\begin{tikzpicture}[scale=0.7]
\tikzset{vertex/.style = {shape=circle,draw,minimum size=3.5em}}
\tikzset{edge/.style = {->,> = latex'}}
\tikzstyle{decision} = [diamond, draw, text badly centered, inner sep=3pt]
\tikzstyle{sq} = [regular polygon,regular polygon sides=4, draw, text badly centered, inner sep=3pt]
\node[vertex] (a) at  (0, 2) {$u$};
\node[vertex] (b) at  (3, 2) {$u+1$};
\node[vertex] (c) at  (0, -1) {$u+d$};
\node[vertex, align=center] (d) at  (3, -1) {\footnotesize{$u+d$} \\ \footnotesize{$+1$}};
\node (e) at (1.5, -3) {Type A};

\node[vertex] (a1) at  (8, 2) {$u$};
\node[vertex] (b1) at  (11, 2) {$u+1$};
\node[vertex] (c1) at  (8, -1) {$u+d$};
\node[vertex, align=center] (d1) at  (11, -1) {\footnotesize{$u+d$} \\ \footnotesize{$+1$}};
\node (e1) at (9.5, -3) {Type B1};

\node[vertex] (a2) at  (16, 2) {$u$};
\node[vertex] (b2) at  (19, 2) {$u+1$};
\node[vertex] (c2) at  (16, -1) {$u+d$};
\node[vertex, align=center] (d2) at  (19, -1) {\footnotesize{$u+d$} \\ \footnotesize{$+1$}};
\node (e2) at (17.5, -3) {Type B2};

\draw[line width=2pt]  (a) -- (b);
\draw[line width=2pt]  (a) -- (c);
\draw[line width=2pt]  (b) -- (d);

\draw[line width=2pt]  (a1) -- (b1);
\draw[line width=2pt]  (a1) -- (c1);
\draw[line width=2pt]  (c1) -- (d1);

\draw[line width=2pt]  (a2) -- (b2);
\draw[line width=2pt]  (b2) -- (d2);
\draw[line width=2pt]  (c2) -- (d2);
\end{tikzpicture}
\end{center}
\caption{Lemma \ref{lm2} proves that none of these structures can occur in a minimal TSP instance (with respect to $n$) where $\sum_{i=1}^d c_i t_i < n-2.$  }\label{fig:bad}\end{figure}
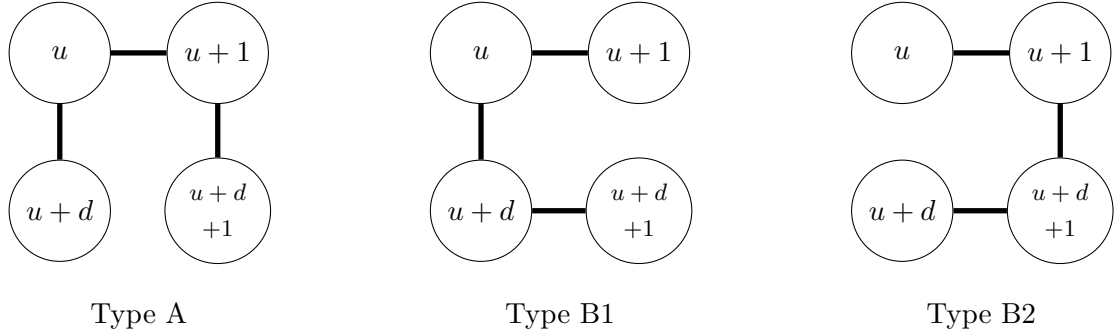

\subsection{Proof of Theorem  \ref{thm:main}}

We can now prove Theorem  \ref{thm:main}.  We consider a hypothetical minimal counterexample to Theorem \ref{thm:main}, i.e., a Hamiltonian cycle where  $$\sum_{i=1}^d c_i t_i < n-2.$$ By Claim \ref{cm1}, $n\geq 8.$  By Lemma \ref{lm2}, our Hamiltonian cycle cannot contain any of  the structures shown in Figure \ref{fig:bad}.  By Lemma \ref{lm1},  $$t_1+2t_d > n+2.$$  We will specifically contradict this claim: we argue that the lack of structures in Figure \ref{fig:bad} forces $t_1+2t_d$ to be small.

To compute $t_1+2t_d$, we look at induced subgraphs of our Hamiltonian cycle on vertices $u, u+1, u+d, u+d+1.$  We call such a subgraph of four vertices a \emph{window}.  Figure \ref{fig:window} shows two natural ways to order the vertices and view a window.  Notice that we can move from one window (e.g. $u=1$) to the next (e.g $u=2$) by rotating (in the left picture) or sliding horizontally (on the right picture). 

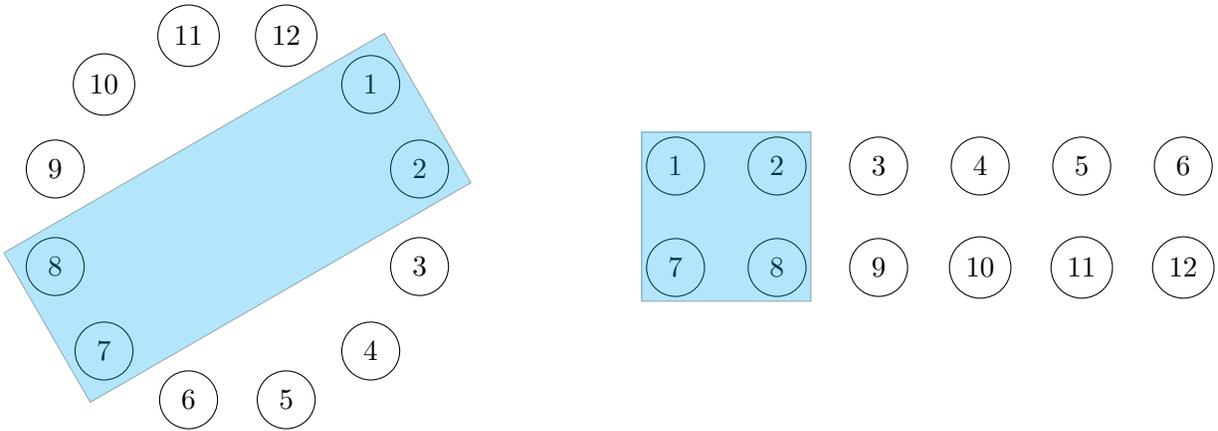
\begin{figure}[h!]
\begin{center}
\begin{tikzpicture}[scale=0.5]

\tikzset{vertex/.style = {shape=circle,draw,minimum size=2em}}
\tikzset{edge/.style = {->,> = latex'}}
\node[draw=none,minimum size=5cm,regular polygon,regular polygon sides=12] (a) {};

\foreach \n [count=\nu from 1, remember=\n as \lastn, evaluate={\nu+\lastn}] in {12, 11, ..., 1} 
\node[vertex, fill=white]at(a.corner \n){$\nu$};

\draw[fill=cyan, opacity=0.3, rotate=30] ($(a.corner 12)+(1, 1)$)  rectangle ($(a.corner 6)-(1, 1)$);

\end{tikzpicture}
\hspace{20mm} 
\begin{tikzpicture}[scale=0.45]
\tikzset{vertex/.style = {shape=circle,draw,minimum size=2em}}
\tikzset{edge/.style = {->,> = latex'}}
\tikzstyle{decision} = [diamond, draw, text badly centered, inner sep=3pt]
\tikzstyle{sq} = [regular polygon,regular polygon sides=4, draw, text badly centered, inner sep=3pt]
\node[vertex] (a) at  (0, 2) {$1$};
\node[vertex] (b) at  (3, 2) {$2$};
\node[vertex] (c) at  (6, 2) {$3$};
\node[vertex] (d) at  (9, 2) {$4$};
\node[vertex] (e) at  (12, 2) {$5$};
\node[vertex] (f) at  (15, 2) {$6$};

\node[vertex] (a1) at  (0, -1) {$7$};
\node[vertex] (b1) at  (3, -1) {$8$};
\node[vertex] (c1) at  (6, -1) {$9$};
\node[vertex] (d1) at  (9, -1) {$10$};
\node[vertex] (e1) at  (12, -1) {$11$};
\node[vertex] (f1) at  (15, -1) {$12$};

\node (fake) at (6, -5.5) {};

\draw[fill=cyan, opacity=0.3] ($(a)+(-1, 1)$)  rectangle ($(b1)-(-1, 1)$);
\end{tikzpicture}
\end{center}
\caption{Two views of a window: a collection of four vertices $u, u+1, u+d, u+d+1.$  In the window shown, $u=1$.}\label{fig:window}\end{figure}

We will count the number of edges of length 1 and $d$ by moving through windows.  If we count the total number of edges in each window with $u=1, 2, 3, ..., d,$ we will exactly count $t_1+2t_d.$ The intuition for this process is shown in Figure \ref{fig:winadd}, which shows exactly those $d$ windows.  Note that every possible length 1 edge is contained in exactly one window (e.g. $1\sim 2$ is only in the window  $1, 2, 7, 8$) while every length-$d$ edge is contained in exactly 2 windows (e.g. $1\sim 7$ is contained in the windows $1, 2, 7, 8$ and $12, 1, 6, 7$).

\begin{figure}[h!]
\begin{center}

\begin{tikzpicture}

\tikzset{vertex/.style = {shape=circle,draw,minimum size=2.5em}}
\tikzset{edge/.style = {->,> = latex'}}
\node[draw=none,minimum size=6cm,regular polygon,regular polygon sides=12] (a) {};

\foreach \n [count=\nu from 1, remember=\n as \lastn, evaluate={\nu+\lastn}] in {12, 11, ..., 1} 
\node[vertex, fill=white]at(a.corner \n){$\nu$};

\draw[fill=cyan, opacity=0.2, rotate=30] ($(a.corner 12)+(.75,.75)$)  rectangle ($(a.corner 6)+(-.75, -.75)$);

\draw[fill=cyan, opacity=0.2, rotate=-30] ($(a.corner 12)+(.75,.75)$)  rectangle ($(a.corner 6)+(-.75, -.75)$);

\draw[fill=cyan, opacity=0.2, rotate=-30] ($(a.corner 10)+(.75,.75)$)  rectangle ($(a.corner 4)+(-.75, -.75)$);

\draw[fill=cyan, opacity=0.2, rotate=-60] ($(a.corner 8)-(-.75,.75)$)  rectangle ($(a.corner 2)-(.75, -.75)$);

\draw[fill=cyan, opacity=0.2] ($(a.corner 3)+(.59,1.53)$)  rectangle ($(a.corner 9)+(-.59, -1.53)$);

\draw[fill=cyan, opacity=0.2, rotate=90] ($(a.corner 3)+(-.59,1.53)$)  rectangle ($(a.corner 9)+(.59, -1.53)$);

\end{tikzpicture}
\end{center}
\caption{Rotating the window from $u=1$ to $u=d$.  Each possible edge of length 1 is in one window, while each possible edge of length $d$ is in two.}\label{fig:winadd}\end{figure}
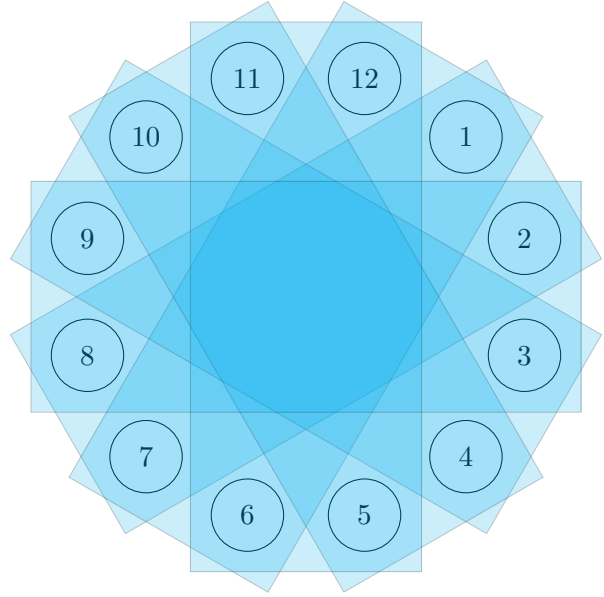

\begin{proof}[Proof (of Theorem \ref{thm:main})]
To formalize our argument, we let $W_u$ denote the window $u, u+1, u+d, u+1+d$.  Figure \ref{fig:window}, e.g., shows the window $W_1=W_7$ when $n=12.$   Let $T_u$ denote the total number of length 1 and $d$ edges within the window $W_u:$ $$T_u = \bi_{u\sim u+1} + \bi_{u+d\sim u+d+1} + \bi_{u\sim u+d} + \bi_{u+1\sim u+1+d}.$$ (Here $\bi_{\circ}$ denotes the indicator function if event $\circ$ occurs, so that, e.g., $\bi_{u\sim u+1}$ is 1 if the edge $\{u, u+1\}$ is in our Hamiltonian cycle, and zero otherwise.)

By Lemma \ref{lm2}, $$T_u \leq 2.$$  Having $T_u=4$ would imply a subtour, and $T_u=3$ would lead to one of the bad structures in Figure \ref{fig:bad}.  

Now we consider rotating the window $d$ times, as in Figure \ref{fig:winadd}.  We note that
\begin{align*}
t_1+2t_d &= T_1 + T_2 + ... + T_d \\
&\leq 2+2+...+2  \hspace{5mm} \text{(d times)}\\
&= 2d\\
&= n.
\end{align*}

The fact that $t_1+2t_d = T_1 + T_2 + ... + T_d$ follows by symmetry:  Any length-1 edge $i\sim i+1$ is included in either $T_i$ (if $i\leq d$) or $T_{i-d}.$  Since each of the $d$ windows contains two length-1 edges, and each of the $n$ length-1 edges is included in exactly one window, each length-1 edge is counted exactly once.  Each length-$d$ edge is analogously included in exactly two windows and counted twice.

This completes our proof, as we have argued that $$t_1+2t_d  \leq n < n+2,$$ contradicting Lemma \ref{lm1}.
\hfill
\end{proof}

\section{Proof of Technical Lemma}\label{sec:lem}
Recall our main technical lemma:
\begin{lm*}[Lemma \ref{lm2}]
Suppose that we have a valid TSP instance where $$\sum_{i=1}^d c_i t_i < n-2,$$ and consider an instance that is minimal with respect to $n$.  Then the counterexample cannot have any of the structures shown in Figure \ref{fig:bad}.   That is, for any $u\in [n],$ a minimal counterexample cannot contain the three edges $\{u+d, u\}, \{u, u+1\}, \{u+1, u+1+d\};$ a minimal counterexample cannot contain the three edges $\{u+1, u\}, \{u, u+d\}, \{u+d, u+1+d\};$ and a minimal counterexample cannot contain the three edges $ \{u, u+1\}, \{u+1, u+1+d\}, \{u+1+d, u+d\}.$
\end{lm*}

Proving this lemma includes involved casework, some of which is deferred to the Appendix.  First, we reduce our work by symmetry, arguing that we only need to consider instances of A and B2 from Figure \ref{fig:bad}.

\begin{cm}
Suppose we have a countexample to Theorem \ref{thm:main} on $n$ vertices containing an instance of B2 from Figure \ref{fig:bad}.  Then there also exists a counterexample  to Theorem \ref{thm:main}  on $n$ vertices containing an instance of B1 from Figure \ref{fig:bad}.
\end{cm}

\begin{proof}
Any counterexample with an instance of B1 can be viewed as a counterexample with an instance of B2 by relabeling the vertices.  The circulant symmetry of the $c_i$ means that we can relabel vertex $i$ as $n-i$ for all $i=1, ..., n$, and then $$\ell_{i, j} = \min \{|i-j|, n-|i-j|\} = \min \{|(n-i)-(n-j)|, n-|(n-i)-(n-j)|\}  = \ell_{n-i, n-j}.$$   Then a counterexample with a B1 implies a counterexample with a B2, and vice versa. \hfill
\end{proof}

\begin{cm}
Suppose we have a countexample to Theorem \ref{thm:main} on $n$ vertices containing an instance of A or B2 from Figure \ref{fig:bad}. Without loss of generality, we may assume that $u=1$.  
\end{cm}

\begin{proof}
Circulant symmetry means that we can relabel vertex $i$ as $i-(u-1)$ for $i=1, ..., n,$ and then  $$\ell_{i, j} = \min \{|i-j|, n-|i-j|\} = \min \{|(i-u+1)-(j-u+1)|, n-|(i-u+1)-(j-u+1)|\}  = \ell_{i-u+1, j-u+1}.$$ \hfill
\end{proof}

Hence, to prove Lemma \ref{lm2}, we need only show that a minimal counterexample cannot have either of the edge sequences shown in Figure \ref{fig:bad2}.

\begin{figure}[h!]
\begin{center}
\begin{tikzpicture}[scale=0.7]
\tikzset{vertex/.style = {shape=circle,draw,minimum size=3.5em}}
\tikzset{edge/.style = {->,> = latex'}}
\tikzstyle{decision} = [diamond, draw, text badly centered, inner sep=3pt]
\tikzstyle{sq} = [regular polygon,regular polygon sides=4, draw, text badly centered, inner sep=3pt]
\node[vertex] (a) at  (0, 2) {$1$};
\node[vertex] (b) at  (3, 2) {$2$};
\node[vertex] (c) at  (0, -1) {$d+1$};
\node[vertex] (d) at  (3, -1) {$d+2$};

\node[vertex] (e) at (5.5, 1) {$j$};
\node[vertex] (f) at (8, 0) {$k$};

\draw[line width=2pt]  (a) -- (b);
\draw[line width=2pt]  (b) -- (d);
\draw[line width=2pt]  (c) -- (d);

\draw[line width=2pt]  (a) to [out=45, in=115] (e);

\draw[line width=2pt]  (c) to [out=-45, in=-115] (f);
\end{tikzpicture}
\hspace{10mm}
\begin{tikzpicture}[scale=0.7]
\tikzset{vertex/.style = {shape=circle,draw,minimum size=3.5em}}
\tikzset{edge/.style = {->,> = latex'}}
\tikzstyle{decision} = [diamond, draw, text badly centered, inner sep=3pt]
\tikzstyle{sq} = [regular polygon,regular polygon sides=4, draw, text badly centered, inner sep=3pt]
\node[vertex] (a) at  (0, 2) {$1$};
\node[vertex] (b) at  (3, 2) {$2$};
\node[vertex] (c) at  (0, -1) {$d+1$};
\node[vertex] (d) at  (3, -1) {$d+2$};

\node[vertex] (e) at (5.5, 1) {$j$};
\node[vertex] (f) at (8, 0) {$k$};

\draw[line width=2pt]  (a) -- (b);
\draw[line width=2pt]  (a) -- (c);
\draw[line width=2pt]  (b) -- (d);

\draw[line width=2pt]  (d) to (e);

\draw[line width=2pt]  (c) to [out=-45, in=-115] (f);
\end{tikzpicture}
\end{center}
\caption{Proving Lemma \ref{lm2} requires showing that neither of these sequences of edges can occur in a minimal counterexample to Theorem \ref{thm:main}.}\label{fig:bad2}\end{figure}
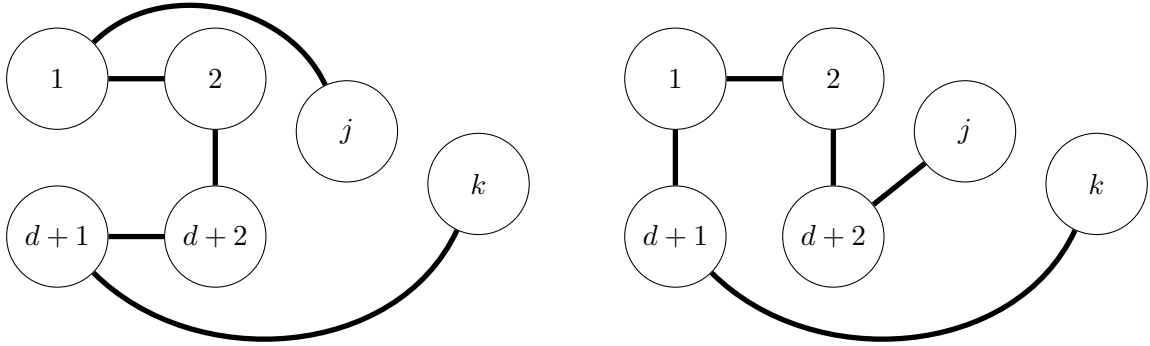

The strategy for both is the same: we show that if we contract the sequence of 5 edges from $j$ to $k$, we attain a counterexample on $4$ fewer vertices.  Figure \ref{fig:cont} shows this contraction process explicitly for the leftmost sequence of edges in Figure \ref{fig:bad2}, which proceeds in two steps.
\begin{enumerate}
\item Delete the vertices $1, 2, d+1, d+2$ and replace the edges $\{j, 1\}, \{1, 2\}, \{2, d+2\}, \{d+2, d+1\}, \{d+1, k\}$ with a single edge $\{j, k\}$, and 
\item  we relabel every other vertex $s$ as $$\begin{cases} s-2, & s\leq d \\ s-4, & s>d.\end{cases}$$
\end{enumerate}

\begin{figure}[h!]
\begin{center}

\begin{tikzpicture}[scale=0.5]
\tikzset{vertex/.style = {shape=circle,draw,minimum size=2.3em}}
\tikzset{edge/.style = {->,> = latex'}}
\tikzstyle{decision} = [diamond, draw, text badly centered, inner sep=3pt]
\tikzstyle{sq} = [regular polygon,regular polygon sides=4, draw, text badly centered, inner sep=3pt]
\node[vertex] (a) at  (0, 2) {$1$};
\node[vertex] (b) at  (3, 2) {$2$};
\node[vertex] (c) at  (6, 2) {$3$};
\node[vertex] (d) at  (9, 2) {$4$};
\node[vertex] (e) at  (12, 2) {$5$};
\node[vertex] (f) at  (15, 2) {$6$};

\node[vertex] (a1) at  (0, -1) {$7$};
\node[vertex] (b1) at  (3, -1) {$8$};
\node[vertex] (c1) at  (6, -1) {$9$};
\node[vertex] (d1) at  (9, -1) {$10$};
\node[vertex] (e1) at  (12, -1) {$11$};
\node[vertex] (f1) at  (15, -1) {$12$};

\node (a2) at  (0, 2) {};
\node[vertex] (c2) at  (6, -6) {$1$};
\node[vertex] (d2) at  (9, -6) {$2$};
\node[vertex] (e2) at  (12, -6) {$3$};
\node[vertex] (f2) at  (15, -6) {$4$};

\node[vertex] (c3) at  (6, -9) {$5$};
\node[vertex] (d3) at  (9, -9) {$6$};
\node[vertex] (e3) at  (12, -9) {$7$};
\node[vertex] (f3) at  (15, -9) {$8$};

\draw[line width=2pt]  (a) to (b);
\draw[line width=2pt]  (b1) to (b);
\draw[line width=2pt]  (a1) to (b1);

\draw[line width=2pt]  (a) to [out=45, in=135] (d);
\draw[line width=2pt]  (a1) to [out=-35, in=-145] (e1);

\draw[line width=1.5pt, densely dotted]  (c) to (d);
\draw[line width=1.5pt, densely dotted]  (c1) to (d1);
\draw[line width=1.5pt, densely dotted]  (c) to (d1);

\draw[line width=1.5pt, densely dotted]  (e) to (f);
\draw[line width=1.5pt, densely dotted]  (e) to (e1);
\draw[line width=1.5pt, densely dotted]  (f1) to (f);

\draw[line width=1.5pt, densely dotted]  (c1) to [out=25, in=155] (f1);

\draw[line width=2pt]  (d2) to (e3);

\draw[line width=1.5pt, densely dotted]  (c2) to (d2);
\draw[line width=1.5pt, densely dotted]  (c3) to (d3);
\draw[line width=1.5pt, densely dotted]  (c2) to (d3);

\draw[line width=1.5pt, densely dotted]  (e2) to (f2);
\draw[line width=1.5pt, densely dotted]  (e2) to (e3);
\draw[line width=1.5pt, densely dotted]  (f3) to (f2);

\draw[line width=1.5pt, densely dotted]  (c3) to [out=25, in=155] (f3);
\end{tikzpicture}
\end{center}
\caption{An example contraction used to prove Lemma \ref{lm2}.  Notice the addition of the edge $\{2, 7\}$ (which corresponds to an edge between vertices 4 and 11 pre-contraction).  The dashed edges do not change.}\label{fig:cont}\end{figure}
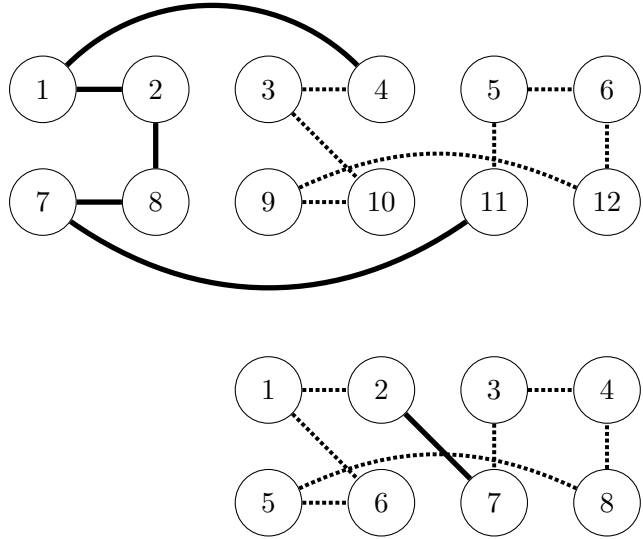

In the rightmost sequence of edges in Figure \ref{fig:bad2}, we proceed analogously, except replacing the edges  $\{j, d+2\}, \{d+2, 2\}, \{2, 1\}, \{1, d+1\}, \{d+1, k\}$ with $\{j, k\}$ in the first step.  

Note that, in both cases, we attain a feasible tour on 4 fewer vertices.  To show that this smaller instance is also a counterexample requires showing, during contraction, we decrease the cost by at least 4.  

Throughout we denote by $s'$ the new label of vertex $s$ after contraction:
$$s'=\begin{cases} s-2, & 2<s\leq d \\ s-4, & s>d+2.\end{cases}$$  We let $$c_{i, j}:=c_{\ell_{i, j}}=\begin{cases} \ell_{i, j},& \ell_{i, j} \text{ odd} \\ d-\ell_{i, j}, & \ell_{i, j} \text{ even} \end{cases}$$ denote the cost of edge $\{i, j\},$ which has length $\ell_{i, j},$ before contraction.  We denote the length of an edge $\{i, j\}$ in the contracted graph $$\ell'_{i, j} = \min\{|i-j|, n-4-|i-j|\}$$ (where  $i$ and $j$ are labels of vertices \emph{in the contracted graph}).  Similarly, we let $$c'_{i, j}=c'_{\ell_{i, j}}=\begin{cases} \ell'_{i, j},& \ell'_{i, j} \text{ odd} \\ d-2-\ell'_{i, j}, & \ell'_{i, j} \text{ even} \end{cases}$$ denote the cost of edge $\{i, j\}$ in the contracted graph.  Again, in this notation,   $i$ and $j$ are labels of vertices \emph{in the contracted graph}.

Note that, by design, any vertex $s$ becomes a vertex $s'$ of the same parity in the contracted graph, and that the length of an edge does not change significantly.  This means, when we contract, we do not need to worry about an edge of odd length becoming an edge of even length, or vice versa, and thus potentially radically changing in cost.  In particular, we first show that, after  contraction, the cost of edges can only decrease. 
\begin{prop}\label{prop:dec}
Let $\{s, t\}$ be an edge with $s, t \notin \{1, 2, d+1, d+2\}.$ Then
$$\ell'_{s', t'} =  \begin{cases} \ell_{s, t}, & 3 \leq s, t \leq d \text{ OR } d+3\leq s, t \leq n \\ \ell_{s, t}-2, & \text{ otherwise. } \end{cases}$$  Moreover
$$c'_{s', t'} = \begin{cases}
 c_{s, t}, & \left(s, t \leq d \text{ OR } d+3\leq s, t\right) \text{ AND } s\not\equiv_2 t \\
 c_{s, t}, & \left( s\leq d \leq d+3 \leq t \text{ OR }  t\leq d \leq d+3 \leq s\right) \text{ AND } s\equiv_2 t \\
c_{s, t} - 2,& \text{ else}. \end{cases}$$
In particular, any edge $\{s, t\}$  with $s, t \notin \{1, 2, d+1, d+2\}$ either gets cheaper or remains the same cost after the contraction process.
\end{prop}

For this proof, it is helpful to consider Figure \ref{fig:blocks}, where the vertices are placed in two groups $\{3, ..., d\}$ and $\{d+3, ..., n\}.$  The above proposition indicates that an edge within a group retains its cost if it is of odd length, and an edge between groups retains its cost if it is of even length.  Other edges -- those within groups and of even length, or those between groups and of odd length -- are lowered in cost by 2.

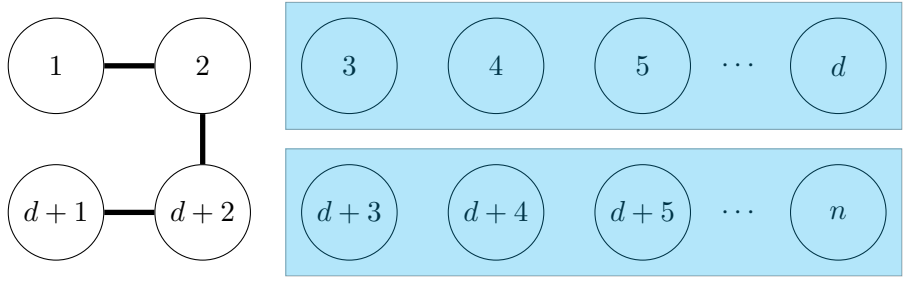
\begin{figure}[h!]
\begin{center}

\begin{tikzpicture}[scale=0.65]
\tikzset{vertex/.style = {shape=circle,draw,minimum size=3.3em}}
\tikzset{edge/.style = {->,> = latex'}}
\tikzstyle{decision} = [diamond, draw, text badly centered, inner sep=3pt]
\tikzstyle{sq} = [regular polygon,regular polygon sides=4, draw, text badly centered, inner sep=3pt]
\node[vertex] (a) at  (0, 2) {$1$};
\node[vertex] (b) at  (3, 2) {$2$};
\node[vertex] (c) at  (6, 2) {$3$};
\node[vertex] (d) at  (9, 2) {$4$};
\node[vertex] (e) at  (12, 2) {$5$};
\node (f) at  (14, 2) {$\cdots$};
\node[vertex] (g) at  (16, 2) {$d$};

\node[vertex] (a1) at  (0, -1) {$d+1$};
\node[vertex] (b1) at  (3, -1) {$d+2$};
\node[vertex] (c1) at  (6, -1) {$d+3$};
\node[vertex] (d1) at  (9, -1) {$d+4$};
\node[vertex] (e1) at  (12, -1) {$d+5$};
\node (f1) at  (14, -1) {$\cdots$};
\node[vertex] (g1) at  (16, -1) {$n$};

\draw[line width=2pt]  (a) to (b);
\draw[line width=2pt]  (b1) to (b);
\draw[line width=2pt]  (a1) to (b1);

\draw[fill=cyan, opacity=0.3] ($(c)+(-1.3, 1.3)$)  rectangle ($(g)-(-1.3, 1.3)$);

\draw[fill=cyan, opacity=0.3] ($(c1)+(-1.3, 1.3)$)  rectangle ($(g1)-(-1.3, 1.3)$);
\end{tikzpicture}
\end{center}
\caption{Placing the vertices into two groups.}\label{fig:blocks}\end{figure}

\begin{proof}
The first part of Proposition \ref{prop:dec} about edge length follows from Figure \ref{fig:blocks}: If $s, t$ are in the same group, then the edge between them does not change in length after contraction.  If they are in different groups, then the edge between them ``goes through'' either $\{1, 2\}$ or $\{d+1, d+2\},$ and after contraction, gets shorter by two.  More formally, if $s, t$ are in the same group, then $s'-s=t'-t$ so that $s'-t'=s-t$ and $|s-t| \leq d-3$ so that $$\ell'_{s', t'} = \min\{|s'-t'|, n-4-|s'-t'|\} = |s'-t'| = |s-t| = \min\{|s-t|, n-|s-t|\} = \ell_{s, t}.$$  Otherwise, without loss of generality let $s$ be in the top group so that $s<t$, $s'<t'$, $s'=s-2,$ and $t'=t-4.$  Then $$|s'-t'| = t'-s' = (t-4)-(s-2)=t-s-2=|s-t|-2,$$ and
$$n-4-|s'-t'| = n-4 - (|s-t|-2) = n - |s-t|-2.$$  Thus 
 $$\ell'_{s', t'} = \min\{|s'-t'|, n-4-|s'-t'|\} = \min\{|s-t|-2, n-|s-t|-2\}  = \min\{|s-t|, n-|s-t|\}-2  = \ell_{s, t}-2.$$

The statement about costs then uses the new lengths and new cost equation.  In two cases:
\begin{itemize}
\item If $s\not\equiv_2 t,$ then 
\begin{align*}
c'_{s', t'} =  \ell'_{s', t'}
 &= \begin{cases}  \ell_{s, t}, &  3 \leq s, t \leq d \text{ OR } d+3\leq s, t \leq n \\\ell_{s, t}-2, &  s\leq d \leq d+3 \leq t \text{ OR }  t\leq d \leq d+3 \leq s \end{cases}\\ 
 &= \begin{cases}  c_{s, t}, &  3 \leq s, t \leq d \text{ OR } d+3\leq s, t \leq n \\ c_{s, t}-2, &  s\leq d \leq d+3 \leq t \text{ OR }  t\leq d \leq d+3 \leq s \end{cases}. \end{align*} 
\item If $s\equiv_2 t,$ then 
\begin{align*}
c'_{s', t'} = d-2 - \ell'_{s', t'}
 &= \begin{cases}  d-2-\ell_{s, t}, &  3 \leq s, t \leq d \text{ OR } d+3\leq s, t \leq n \\ d-2 - (\ell_{s, t}-2), &  s\leq d \leq d+3 \leq t \text{ OR }  t\leq d \leq d+3 \leq s \end{cases}\\ 
 &= \begin{cases}  c_{s, t}-2, &  3 \leq s, t \leq d \text{ OR } d+3\leq s, t \leq n \\ c_{s, t}, &  s\leq d \leq d+3 \leq t \text{ OR }  t\leq d \leq d+3 \leq s \end{cases}. \end{align*} 
\end{itemize}

These cases give the claimed results. \hfill
\end{proof}

The only other change, after contraction, is that we contract five edges into one edge $\{j', k'\}.$  The following propositions account for the change in cost after these contractions.

\begin{prop}\label{prop:cont1}
Consider the first case, where we contract $\{j, 1\}, \{1, 2\}, \{2, d+2\}, \{d+2, d+1\}, \{d+1, k\}$ into $\{j', k'\}.$  Then 
\begin{equation}\label{eq:cont1}
c_{j, 1} + c_{1, 2} + c_{2, d+2} + c_{d+2, d+1} + c_{d+1, k} - c'_{j', k'} \geq 4.
\end{equation}
Hence the  cost of the tour resulting from the contraction goes down in cost by at least 4. 
\end{prop}

\begin{prop}\label{prop:cont2}
Consider the second case, where we contract $\{j, d+2\}, \{d+2, 2\}, \{2, 1\}, \{1, d+1\}, \{d+1, k\}$ into $\{j', k'\}.$  Then 
\begin{equation} \label{eq:cont2}
c_{j, d+2} + c_{d+2, 2} + c_{2, 1} + c_{1, d+1} + c_{d+1, k} - c'_{j', k'} \geq 2
\end{equation}
and is even.  Moreover, if
$$c_{j, d+2} + c_{d+2, 2} + c_{2, 1} + c_{1, d+1} + c_{d+1, k} - c'_{j', k'} =2,$$ then there must have been at least one edge $\{s, t\}$ in the cycle with  $s, t \notin \{1, 2, d+1, d+2\}$ such that $c'_{s', t'} = c_{s, t}-2.$
In either case, the cost of the tour resulting from the contraction goes down in cost by at least 4. 
\end{prop}

Propositions \ref{prop:cont1} and  \ref{prop:cont2} imply that a tour with any of the structures indicated in Lemma \ref{lm2}
and Figure \ref{fig:bad} can be contracted to attain a tour on 4 fewer vertices, and whose aggregate cost is at least 4 cheaper.  They thus complete the proof of Lemma \ref{lm2}.

For the sake of full precision we provide analytic formulas for $c_{j, 1} + c_{1, 2} + c_{2, d+2} + c_{d+2, d+1} + c_{d+1, k} - c'_{j', k'}$ and $c_{j, d+2} + c_{d+2, 2} + c_{2, 1} + c_{1, d+1} + c_{d+1, k} - c'_{j', k'}.$  Doing so, however, involves substantial casework.  We thus defer it to the appendix.  Up to that casework, the propositions follow quickly.

\begin{proof}[Proof (of Proposition \ref{prop:cont1})]
Note that $$c_{j, 1} + c_{1, 2} + c_{2, d+2} + c_{d+2, d+1} + c_{d+1, k} - c'_{j', k'} = c_{j, 1}  + c_{d+1, k} - c'_{j', k'} + 2.$$ Casework in the appendix shows that $$c_{j, 1}  + c_{d+1, k} - c'_{j', k'} \geq 2,$$ which completes the proof. \hfill
\end{proof}

\begin{proof}[Proof (of Proposition \ref{prop:cont2})]
Note that $$c_{j, d+2} + c_{d+2, 2} + c_{2, 1} + c_{1, d+1} + c_{d+1, k} - c'_{j', k'} = c_{j, d+2} + c_{d+1, k} - c'_{j', k'} + 1.$$ 
To show that this equation is even, we want to show that $c_{j, d+2} + c_{d+1, k} - c'_{j', k'}$ is odd.  This  follows because $$c_{j, d+2} \equiv_2 c_{d+1, k} \text{ if and only if } c'_{j', k'} \equiv_2 1.$$  (This can also be seen by considering the four cases depending on $j$ and $k$s individual parity.)

 Casework in the appendix shows that $$c_{j, d+2} + c_{d+1, k} - c'_{j', k'} \geq 1,$$ and that equality can hold in exactly four cases: 
\begin{itemize}
\item $j, k$ are even and $j, k \leq d$
\item $j, k$ are odd and $j, k \geq d+3$
\item $j$ is odd, $k$ is even, $j\geq d+3$, and $k\leq d$
\item $j$ is even, $k$ is odd, $j\leq d$, and $k\geq d+3.$
\end{itemize}

Recall from Proposition \ref{prop:dec} that, for an edge $\{s, t\}$ with $s, t \notin \{1, 2, d+1, d+2\},$ either $c'_{s', t'}=c_{s, t}$ or $c'_{s', t'}=c_{s, t}-2$.  The existence of a single edge where the cost decreases is sufficient to complete the proof of Proposition \ref{prop:cont2}.  Hence, we are concerned only with the case where $c'_{s', t'}=c_{s, t}$ for every edge $\{s, t\}$ in the tour (outside of  $\{j, d+2\}, \{d+2, 2\}, \{2, 1\}, \{1, d+1\}, \{d+1, k\}$).  From Proposition  \ref{prop:dec}, these are exactly the edges $\{s, t\}$ where 
\begin{itemize}
\item $\left(s, t \leq d \text{ OR } d+3\leq s, t\right) \text{ AND } s\not\equiv_2 t$
\item $  \left( s\leq d \leq d+3 \leq t \text{ OR }  t\leq d \leq d+3 \leq s\right) \text{ AND } s\equiv_2 t.$
\end{itemize}

Referring again to the groups indicated in Figure \ref{fig:blocks}, these are exactly the within-group edges of odd length and the across-group edges of even length. We now break the vertices $S=\{3, 4, ..., d, d+3, d+4, ..., n-1, n\}$ into four groups, as indicated in Figure \ref{fig:groups}: $$\{v\in S: v\leq d, v\text{ odd}\}, \{v\in S: v\leq d, v\text{ even}\}, \{v\in S: v> d, v\text{ odd}\}, \text{ and }\{v\in S: v> d, v\text{ even}\}.$$ We let $S_1$ denote the groups where $\{v\in S: v\leq d, v\text{ odd}\}$ and $\{v\in S: v> d, v\text{ even}\}.$  We let $S_2$ denote the other groups.  Let $G^{\circ}$ be the graph shown in Figure \ref{fig:groups}, where there is one vertex for each group, and edges between groups in opposite sets $S_1, S_2.$

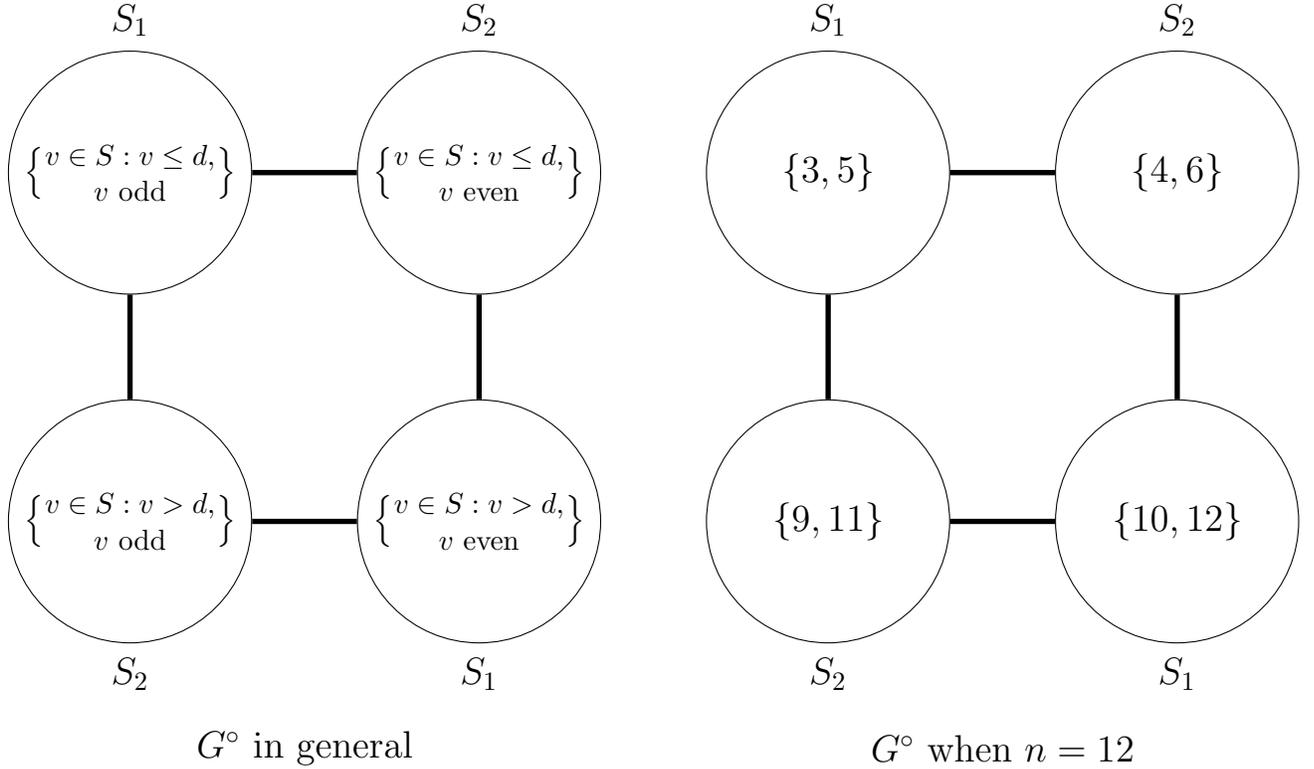
\begin{figure}[h!]
\begin{center}
\hspace{-5mm}
\begin{tikzpicture}[scale=0.58]
\tikzset{vertex/.style = {shape=circle,draw,minimum size=8.4em}}
\tikzset{edge/.style = {->,> = latex'}}
\tikzstyle{decision} = [diamond, draw, text badly centered, inner sep=3pt]
\tikzstyle{sq} = [regular polygon,regular polygon sides=4, draw, text badly centered, inner sep=3pt]
\node[vertex] (a) at  (0, 2) {$\left\{ \parbox[c]{23mm}{\centering $v\in S: v\leq d,$ $v \text{ odd}$} \right\}$};
\node[vertex] (b) at  (8, 2) {$\left\{ \parbox[c]{23mm}{\centering $v\in S: v\leq d,$ $v \text{ even}$} \right\}$};
\node[vertex] (a1) at  (0, -6) {$\left\{ \parbox[c]{23mm}{\centering $v\in S: v> d,$ $v \text{ odd}$} \right\}$};
\node[vertex] (b1) at  (8, -6) {$\left\{ \parbox[c]{23mm}{\centering $v\in S: v> d,$ $v \text{ even}$} \right\}$};
\node (a0) at (0, 5.5) {\Large{$S_1$}};
\node (b0) at (8, 5.5) {\Large{$S_2$}};
\node (a0) at (0, -9.5) {\Large{$S_2$}};
\node (b0) at (8, -9.5) {\Large{$S_1$}};
\node (label) at (4, -11.2) {\Large{$G^{\circ}$ in general}};

\draw[line width=2pt]  (a) to (b);
\draw[line width=2pt]  (a) to (a1);
\draw[line width=2pt]  (b1) to (b);
\draw[line width=2pt]  (a1) to (b1);

\node[vertex] (c) at  (16, 2) {\Large{$\left\{ 3, 5 \right\}$}};
\node[vertex] (d) at  (24, 2) {\Large{$\left\{ 4, 6 \right\}$}};
\node[vertex] (c1) at  (16, -6) {\Large{$\left\{ 9, 11 \right\}$}};
\node[vertex] (d1) at  (24, -6) {\Large{$\left\{ 10, 12\right\}$}};
\node (c0) at (16, 5.5) {\Large{$S_1$}};
\node (d0) at (24, 5.5) {\Large{$S_2$}};
\node (c0) at (16, -9.5) {\Large{$S_2$}};
\node (d0) at (24, -9.5) {\Large{$S_1$}};
\node (label1) at (20, -11.2) {\Large{$G^{\circ}$ when $n=12$}};

\draw[line width=2pt]  (c) to (d);
\draw[line width=2pt]  (c) to (c1);
\draw[line width=2pt]  (d1) to (d);
\draw[line width=2pt]  (c1) to (d1);

\end{tikzpicture}
\end{center}
\caption{Placing the vertices into four groups to attain the graph $G^\circ$: in general (on the left) and explicitly when $n=12$ (on the right)}\label{fig:groups}\end{figure}

Note that all four groups have equal size $(n-4)/4.$  The edges between groups in $G^{\circ}$, moreover, indicate exactly the cases where $c_{s, t}=c'_{s', t'}.$  For example, consider an edge $\{s, t\}$ where $s, t \leq d$ and $s\not\equiv_2 t.$  This corresponds from moving between the groups $\{v\in S: v\leq d, v\text{ odd}\}$ and $\{v\in S: v\leq d, v\text{ even}\},$ which is indicated by the top horizontal edge.

Note also that the only cases where $$c_{j, d+2} + c_{d+1, k} - c'_{j', k'} = 1$$ correspond to cases where both $j, k\in S_2.$  Suppose there exists a Hamiltonian tour (in the original, precontracted graph on $n$ vertices) where $c_{j, d+2} + c_{d+1, k} - c'_{j', k'} = 1$ and $c'_{s', t'}=c_{s, t}$ for every other edge.  That is, consider the only potential type of tour whose cost does not go down by at least 4 after contraction.

We trace the $n-4$ vertices that connect $j$ to $k$ through $S$ in such a tour. Tracing the vertices we visit in the  graph $G^{\circ}$ corresponds to a walk in $G^{\circ}$ that:
\begin{enumerate}
\item Starts and ends at one of the $S_2$ vertices and (since $c_{j, d+2} + c_{d+1, k} - c'_{j', k'} = 1$),  
\item Visits each of the four nodes in $G^{\circ}$ an equal number of times (since each node in $G^{\circ}$ has an equal number of vertices),
\item and only uses the four edges in $G^{\circ}$.  I.e., the walk never uses an edge between the two $S_2$ vertices, or an edge between the two $S_1$ vertices (since we only use  edges where $c_{s, t}=c'_{s', t'},$ which are exactly the edges in  $G^{\circ}$).
\end{enumerate}

No such walk can exist: The first and third criteria indicate that the traced walk will look like $S_2, S_1, S_2, S_1, ...., S_2, S_1, S_2.$  Hence, if it visits vertices in $S_1$ $k$ times, it visits vertices in $S_2$ $k+1$ times. Thus we will never satisfy the second criteria, that we visit exactly as many $S_1$ vertices as we visit $S_2$ vertices.
\hfill
\end{proof}

\section{Inequality (\ref{eq:main}) is Facet-Defining}\label{sec:facet}

In the previous two sections, we showed that  the circlet inequalities were  valid for the TSP. We now also prove that it is facet defining for the Symmetric Traveling Salesman Polytope  $STSP(n)$.  Recall that $\chi_H\in\{0, 1\}^{|E|}$ denotes the incidence vector of a Hamiltonian cycle $H$ on $K_n$.  In terms of these incidence vectors, $STSP(n)$ is  $$STSP(n)=\text{conv}\{\chi_H: H \text{ is a Hamiltonian cycle on } K_n\}\subset \R_{\geq 0}^E.$$  To show that a valid TSP inequality is facet-defining for $STSP(n)$, we follow Theorem 3.7 in Naddef and Rinaldi \cite{Nad92}: 
The dimension of $STSP(n)$ is $|E|-|V|=\f{n(n-3)}{2};$ to show that a valid inequality is facet defining for the TSP, we must thus find $n(n-3)/2$ Hamiltonian cycles for which the circlet inequality is tight and whose incidence vectors are linearly independent.

To show that this inequality is valid, we consider $d-1$ distinct types of tours.  These tours will contain edges of length 1 and $d$, and up to one edge of a different length $k$.  We will index the tours by this extra edge length, and count the number of linearly independent tours of each type.  For intuition, consider the case when $n=8$.  We consider three types of tours shown in Figure \ref{fig:8case}.

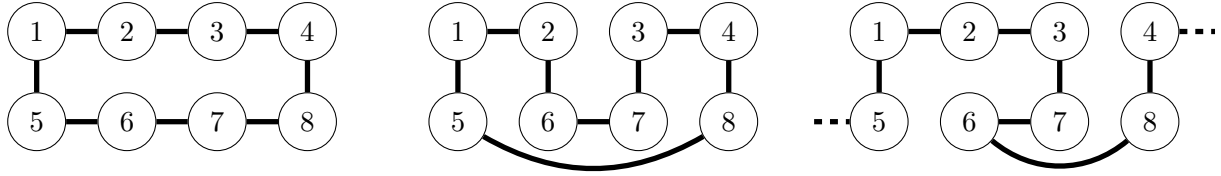
\begin{figure}[h!]
\begin{center}
\begin{tikzpicture}[scale=0.4]
\tikzset{vertex/.style = {shape=circle,draw,minimum size=2em}}
\tikzset{edge/.style = {->,> = latex'}}
\tikzstyle{decision} = [diamond, draw, text badly centered, inner sep=3pt]
\tikzstyle{sq} = [regular polygon,regular polygon sides=4, draw, text badly centered, inner sep=3pt]
\node[vertex] (a) at  (0, 2) {$1$};
\node[vertex] (b) at  (3, 2) {$2$};
\node[vertex] (c) at  (6, 2) {$3$};
\node[vertex] (d) at  (9, 2) {$4$};

\node[vertex] (a1) at  (0, -1) {$5$};
\node[vertex] (b1) at  (3, -1) {$6$};
\node[vertex] (c1) at  (6, -1) {$7$};
\node[vertex] (d1) at  (9, -1) {$8$};

\draw[line width=2pt]  (a) to (b);
\draw[line width=2pt]  (c) to (b);
\draw[line width=2pt]  (c) to (d);
\draw[line width=2pt]  (a1) to (b1);
\draw[line width=2pt]  (c1) to (b1);
\draw[line width=2pt]  (c1) to (d1);
\draw[line width=2pt]  (a) to (a1);
\draw[line width=2pt]  (d) to (d1);

\node[vertex] (a2) at  (14, 2) {$1$};
\node[vertex] (b2) at  (17, 2) {$2$};
\node[vertex] (c2) at  (20, 2) {$3$};
\node[vertex] (d2) at  (23, 2) {$4$};

\node[vertex] (a3) at  (14, -1) {$5$};
\node[vertex] (b3) at  (17, -1) {$6$};
\node[vertex] (c3) at  (20, -1) {$7$};
\node[vertex] (d3) at  (23, -1) {$8$};

\draw[line width=2pt]  (a2) to (b2);
\draw[line width=2pt]  (c2) to (d2);
\draw[line width=2pt]  (c3) to (b3);
\draw[line width=2pt]  (a2) to (a3);
\draw[line width=2pt]  (b2) to (b3);
\draw[line width=2pt]  (c2) to (c3);
\draw[line width=2pt]  (d2) to (d3);
\draw[line width=2pt]  (a3) to [out=-30, in=-150] (d3);

\node[vertex] (a4) at  (28, 2) {$1$};
\node[vertex] (b4) at  (31, 2) {$2$};
\node[vertex] (c4) at  (34, 2) {$3$};
\node[vertex] (d4) at  (37, 2) {$4$};
\node (fake) at (39.5, 2) {};
\node (fake2) at (25.5, -1) {};
\node[vertex] (a5) at  (28, -1) {$5$};
\node[vertex] (b5) at  (31, -1) {$6$};
\node[vertex] (c5) at  (34, -1) {$7$};
\node[vertex] (d5) at  (37, -1) {$8$};

\draw[line width=2pt]  (a4) to (b4);
\draw[line width=2pt]  (c4) to (b4);
\draw[line width=2pt]  (b5) to (c5);
\draw[line width=2pt]  (a4) to (a5);
\draw[line width=2pt]  (c4) to (c5);
\draw[line width=2pt]  (d4) to (d5);

\draw[line width=2pt, dashed]  (d4) to (fake);
\draw[line width=2pt, dashed]  (fake2) to (a5);

\draw[line width=2pt]  (b5) to [out=-40, in=-140] (d5);
\end{tikzpicture}
\end{center}
	\caption{Three types of tours when $n=8:$ a tour that uses only edges of length $1$ and $d$, a tour that uses one edge of length $3$, and a tour that uses one edge of length $2$.  Note that we can ``rotate'' the labels of the each tour (e.g. replacing every vertex $v$ with $v+1 \mod 8$) to attain distinct tours of the same type.  Doing so traces out four distinct tours of the leftmost type, eight distinct tours of the middle type, and eight distinct tours of the rightmost type.}\label{fig:8case}\end{figure}

\begin{figure}[h!]
\begin{center}
\begin{tikzpicture}
\tikzset{vertex/.style = {shape=circle,draw,minimum size=2.5em}}
\tikzset{edge/.style = {->,> = latex'}}
\node[draw=none,minimum size=4cm,regular polygon,regular polygon sides=8] (a) at (0, 0) {};

\draw[line width=2pt] (a.corner 4) to (a.corner 8);
\draw[line width=2pt] (a.corner 1) to (a.corner 5);

\draw[line width=2pt] (a.corner 7) to (a.corner 8);
\draw[line width=2pt] (a.corner 7) to (a.corner 6);
\draw[line width=2pt] (a.corner 5) to (a.corner 6);

\draw[line width=2pt] (a.corner 4) to (a.corner 3);
\draw[line width=2pt] (a.corner 2) to (a.corner 3);
\draw[line width=2pt] (a.corner 2) to (a.corner 1);

\foreach \n [count=\nu from 1, remember=\n as \lastn, evaluate={\nu+\lastn}] in {8, 7, ..., 1} 
\node[vertex, fill=white]at(a.corner \n){$\nu$};

\node[draw=none,minimum size=4cm,regular polygon,regular polygon sides=8] (b)  at (8, 0) {};

\draw[line width=2pt] (b.corner 4) to (b.corner 8);
\draw[line width=2pt] (b.corner 1) to (b.corner 5);

\draw[line width=2pt] (b.corner 7) to (b.corner 8);
\draw[line width=2pt] (b.corner 7) to (b.corner 6);
\draw[line width=2pt] (b.corner 5) to (b.corner 6);

\draw[line width=2pt] (b.corner 4) to (b.corner 3);
\draw[line width=2pt] (b.corner 2) to (b.corner 3);
\draw[line width=2pt] (b.corner 2) to (b.corner 1);

\foreach \n [count=\nu from 1, remember=\n as \lastn, evaluate={\j=int(mod(\nu, 8)+1)}] in {8, 7, ..., 1} 
\node[vertex, fill=white]at(b.corner \n){$\j$};

\end{tikzpicture}
\end{center}
	\caption{One version of of the leftmost tour from Figure \ref{fig:8case} and one rotation, where each vertex label $v$ is replaced with $v+1 \mod 8$.  This corresponds to ``rotating'' the labels 45 degrees counter-clockwise.}\label{fig:exRot}\end{figure}
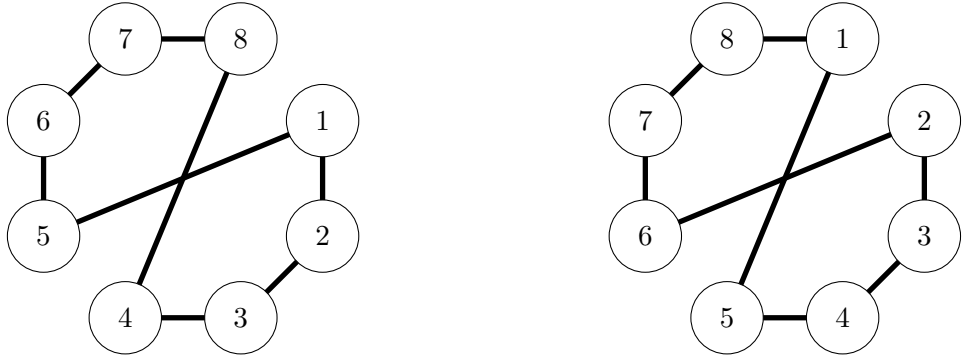

By construction, the circlet inequality is tight for all tours in Figure \ref{fig:8case}.  We can attain $n=8$ versions of the second and third tour type by ``rotating'' the vertex labels (replacing every vertex $v$ with $v+1 \mod 8$ constitutes one rotation); by circulant symmetry, each rotated tour is again tight for the circlet inequality.  We can similarly get $n/2=4$ copies of the first tour (as after four rotations, we return to the original labeling).  In total this gives $$4+2*8 = 20 = \f{8*5}{2}=\f{n(n-3)}{2}$$ tours for which the inequality is tight.

When we generalize this argument, we will again consider $n/2$ tours of the leftmost type (with 2 edges of length $d$, and $n-2$ edges of length $1$).  Then, for each of the $\f{n}{2}-2$ possible values of $k\in\{2, 3, ..., \f{n}{2}-1\}$, we will analgously find a tour using exactly one edge of cost $k$ (and all other edges of length $1$ or $d$). 
\begin{thm}\label{thm:tight}
Let $4| n$.  Then the circlet inequalities are facet defining for $STSP(n)$.
\end{thm}

\begin{proof}
We follow the intuition outlined above:
\begin{itemize}
\item There are $d=\f{n}{2}$ tours with two edges of length $d$ and $n-2$ edges of length $1$: One such tour uses the two length-$d$ edges $\{1, d+1\}$ and $\{d, n\},$ and connects them via $\{1, 2\}, \{2, 3\}, ...,\{d-1, d\}$ and $\{d+1, d+2\}, \{d+2, d+3\}, ..., \{n-1, n\}.$
 This tour costs $$0(2)+(n-2)=n-2,$$ and is tight for the circlet inequality.    We then can rotate this tour by adding a constant $m$ to the label of every vertex, for $m=1, 2, ..., d-1$, and, by circulant symmetry, will attain another tight tour.
\item For $k\in \{3, ..., d\}$, we find $n$ tours that each has a unique edge of cost $k-1$.  These will, moreover, have $k$ edges of length $d$ and $n-k-1$ edges of length 1.  Their total cost will be $$(k-1)+0(k)+(n-k-1)=n-2,$$ and they are indeed tight for the circlet inequality.  To do so, we follow the previous process: we start with a tour with a single edge of cost $k-1,$ and rotate it.  For these tours, however, we are able to rotate them by adding any constant $m=1, 2, ..., n-1$ to the label of every vertex (as we trace every edge of cost $k-1$).
\begin{itemize}
\item When $k$ is even, we use the tour type shown in Figure \ref{fig:keven}. Note that, by construction, it uses $k$ edges of length $d$ and $n-k-1$ edges of length 1.  The remaining edge is of length $k-1$ and, since $k$ is even, it is of cost $k-1$.
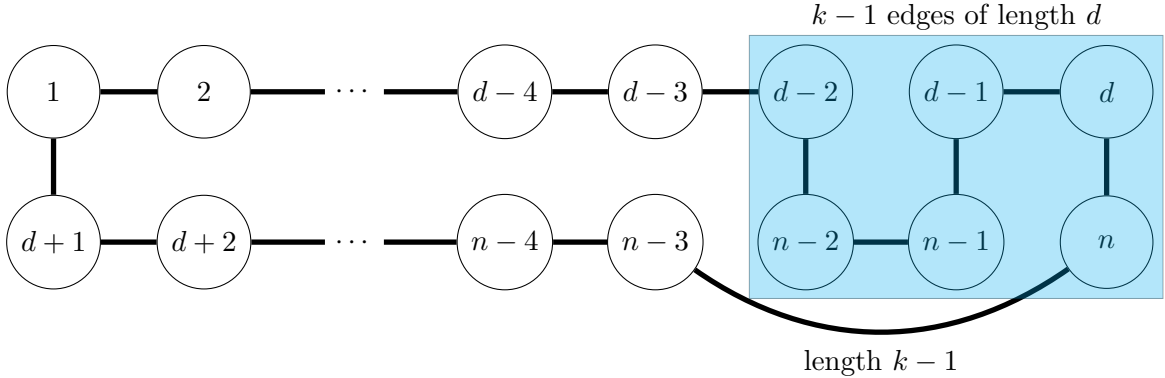
\begin{figure}[h!]
	\begin{center}
\begin{tikzpicture}[scale=0.5]
\tikzset{vertex/.style = {shape=circle,draw,minimum size=3.2em}}
\tikzset{edge/.style = {->,> = latex'}}
\tikzstyle{decision} = [diamond, draw, text badly centered, inner sep=3pt]
\tikzstyle{sq} = [regular polygon,regular polygon sides=4, draw, text badly centered, inner sep=3pt]
\node[vertex] (a) at  (0, 2) {$1$};
\node[vertex] (b) at  (4, 2) {$2$};
\node (c) at  (8, 2) {$\cdots$};
\node[vertex] (d) at  (12, 2) {$d-4$};
\node[vertex] (e) at  (16, 2) {$d-3$};
\node[vertex] (f) at  (20, 2) {$d-2$};
\node[vertex] (g) at  (24, 2) {$d-1$};
\node[vertex] (h) at  (28, 2) {$d$};

\node[vertex] (a1) at  (0, -2) {$d+1$};
\node[vertex] (b1) at  (4, -2) {$d+2$};
\node (c1) at  (8, -2) {$\cdots$};
\node[vertex] (d1) at  (12, -2) {$n-4$};
\node[vertex] (e1) at  (16, -2) {$n-3$};
\node[vertex] (f1) at  (20, -2) {$n-2$};
\node[vertex] (g1) at  (24, -2) {$n-1$};
\node[vertex] (h1) at  (28, -2) {$n$};

\draw[line width=2pt]  (a) -- (b);
\draw[line width=2pt]  (b) -- (c);
\draw[line width=2pt]  (d) -- (c);
\draw[line width=2pt]  (d) -- (e);
\draw[line width=2pt]  (e) -- (f);
\draw[line width=2pt]  (g) -- (h);

\draw[line width=2pt]  (a1) -- (b1);
\draw[line width=2pt]  (b1) -- (c1);
\draw[line width=2pt]  (d1) -- (c1);
\draw[line width=2pt]  (d1) -- (e1);
\draw[line width=2pt]  (f1) -- (g1);

\draw[line width=2pt]  (a) -- (a1);
\draw[line width=2pt]  (f1) -- (f);
\draw[line width=2pt]  (g) -- (g1);
\draw[line width=2pt]  (h) -- (h1);

\draw[line width=2pt]  (e1) to [out=-35, in=-145] (h1);

\draw[fill=cyan, opacity=0.3] ($(f)+(-1.5, 1.5)$)  rectangle ($(h1)-(-1.5, 1.5)$);

\node (fake) at (24, 4) {$k-1$ edges of length $d$};
\node (fake) at (22, -5.2) {length $k-1$};
\end{tikzpicture}
	\end{center}
	\caption{Tight tours using exactly one edge of length $k$ when $k$ is even.  In this example, $k=4$.}\label{fig:keven}\end{figure}

\item When $k$ is odd, we use the tour type shown in Figure \ref{fig:kodd}. Again, by construction, it uses $k$ edges of length $d$ and $n-k-1$ edges of length 1.  The remaining edge is of length $d-(k-1)$ and, since $k$ is odd, it is of cost $d-\left(d-\left(k-1\right)\right)=k-1.$
\begin{figure}[h!]
\begin{center}

\begin{tikzpicture}[scale=0.5]
\tikzset{vertex/.style = {shape=circle,draw,minimum size=3.2em}}
\tikzset{edge/.style = {->,> = latex'}}
\tikzstyle{decision} = [diamond, draw, text badly centered, inner sep=3pt]
\tikzstyle{sq} = [regular polygon,regular polygon sides=4, draw, text badly centered, inner sep=3pt]
\node[vertex] (a) at  (0, 2) {$1$};
\node[vertex] (b) at  (4, 2) {$2$};
\node (c) at  (8, 2) {$\cdots$};
\node[vertex] (d) at  (12, 2) {$d-4$};
\node[vertex] (e) at  (16, 2) {$d-3$};
\node[vertex] (f) at  (20, 2) {$d-2$};
\node[vertex] (g) at  (24, 2) {$d-1$};
\node[vertex] (h) at  (28, 2) {$d$};

\node[vertex] (a1) at  (0, -2) {$d+1$};
\node[vertex] (b1) at  (4, -2) {$d+2$};
\node (c1) at  (8, -2) {$\cdots$};
\node[vertex] (d1) at  (12, -2) {$n-4$};
\node[vertex] (e1) at  (16, -2) {$n-3$};
\node[vertex] (f1) at  (20, -2) {$n-2$};
\node[vertex] (g1) at  (24, -2) {$n-1$};
\node[vertex] (h1) at  (28, -2) {$n$};

\draw[line width=2pt]  (a) -- (b);
\draw[line width=2pt]  (b) -- (c);
\draw[line width=2pt]  (d) -- (c);
\draw[line width=2pt]  (e) -- (d);
\draw[line width=2pt]  (f) -- (g);

\draw[line width=2pt]  (a1) -- (b1);
\draw[line width=2pt]  (b1) -- (c1);
\draw[line width=2pt]  (d1) -- (c1);
\draw[line width=2pt]  (f1) -- (e1);
\draw[line width=2pt]  (g1) -- (h1);

\draw[line width=2pt]  (a) -- (a1);
\draw[line width=2pt]  (e1) -- (e);
\draw[line width=2pt]  (f1) -- (f);
\draw[line width=2pt]  (g) -- (g1);
\draw[line width=2pt]  (h) -- (h1);

\draw[line width=2pt]  (d1) to [out=-35, in=-125] (31.7, -1.5)[out = 67, in=0]  to  (h);
\draw[fill=cyan, opacity=0.3] ($(e)+(-1.5, 1.5)$)  rectangle ($(h1)-(-1.5, 1.5)$);

\node (fake) at (22, 4) {$k-1$ edges of length $d$};
\node (fake) at (22, -6.5) {length $d-\left(k-1\right)$};
\end{tikzpicture}
\end{center}
	\caption{Tight tours using exactly one edge of length $k$ when $k$ is odd.  In the figure, $k=5$.}\label{fig:kodd}\end{figure}
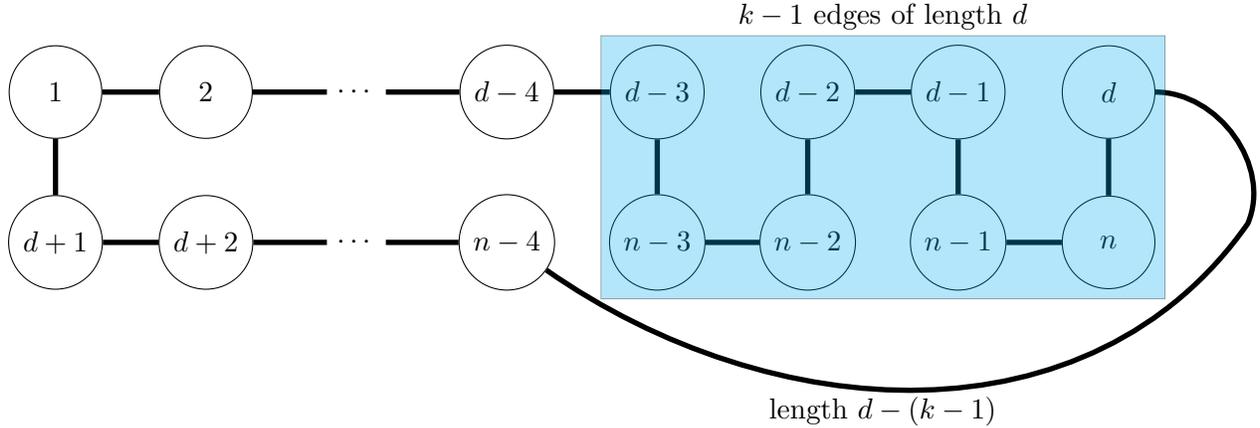
\end{itemize}
\end{itemize}
In total, these give $$\f{n}{2} + \left(\f{n}{2}-2\right)n = \f{n}{2} + \f{n(n-4)}{2} = \f{n(n-3)}{2}$$ tight tours.  Their characteristic vectors are linearly independent for the  same reason as in the $n=8$ case. All of the tours with a unique edge of cost $k-1$  cannot be used in non-trivial linear dependency: if edge $e$ is used uniquely in a given tour $H$, then no other characteristic vector places and weight on edge $e$ and so the coefficient of $\chi_H$ must be zero in any non-trivial linear dependency.   From there, the remaining tours with two edges of length $d$ and $n-2$ edges of length $1$ can also not be used in any non-trivial linear dependency, as each contains a unique edge $e$ of length $d$ not used by any of the remaining tours.  Hence  any linear dependency must be trivial. \hfill
\end{proof}

\section{The Strength of the Circlet Inequalities}\label{sec:Strength}
Goemans \cite{Goe95b} provides a way of evaluating the strength of facet-defining inequalities for the TSP.  This notion is with respect to the \emph{Graphic Traveling Salesman Problem} $GTSP(n)$.  Whereas $STSP(n)$ is the convex hull of incidence vectors of Hamiltonian cycles on $K_n$, $GTSP(n)$ is  convex hull of incidence vectors of Eulerian sub(multi)graphs on $K_n$: vectors $\chi_S\in \N^E$ where $S$ is a multiset of edges in $K_n$ such that the multigraph $([n], S)$ is connected and every vertex has even degree.  Unlike $STSP(n)$, $GTSP(n)$ is full dimensional.  Similarly, $\chi_H \in GTSP(n)$ for any Hamiltonian cycle $H$ on $K_n.$  

Goemans' \cite{Goe95b} definition of strength of a TSP inequality is relative to the  prototypical TSP relaxation, the subtour LP, and whose feasible region we abbreviate as $SP(n)$.  Theorem 2.11 in Goemans \cite{Goe95b} shows that any nontrivial inequality in \emph{tight triangular form} that is facet-defining for $STSP(n)$ defines a facet of $GTSP(n)$. Given such an inequality $ax\geq b$, its \emph{strength} relative to the subtour elimination polyhedron $SP(n)$ is $$\f{b}{\min\{ax:x\in SP(n)\}}.$$  

To compute the strength of the circlet inequality, we must first convert it to tight triangular form.  An inequality $fx\geq f_0$ is in tight triangular form if:
\begin{itemize}
\item $f_{ij} + f_{jk} \geq f_{ik}$ for all distinct triples $i, j, k$, and
\item for all $j\in V$, there exist some $i, k \in V\backslash\{j\}$ such that $f_{i, j}+f_{j, k}=f_{i, k}.$  
\end{itemize}
See Section 2 of  Goemans \cite{Goe95b}.  One converts an inequality to tight triangular form by adding/substracting multiples of the degree constraints. 

In tight triangular form, the circlet inequalities can be stated as $$fx\geq \f{n^2}{2}-n-2, \hspace{5mm} f_e=\begin{cases} d-2+\ell_e, & \ell_e \text{ odd} \\ d-2+(d-\ell_e), & \ell_e \text{ even.} \end{cases}$$
We can rewrite this in the form of Inequality (\ref{eq:main}), as
\begin{equation}\label{eq:TTF}\sum_{i=1}^d f_i t_i \geq \f{n^2}{2}-n-2, \hspace{5mm} f_i=\begin{cases} d-2+i, & i \text{ odd} \\ d-2+(d-i), & i \text{ even.} \end{cases} \end{equation}
These are obtained from Inequality (\ref{eq:main}) by adding $\f{1}{2}(d-2)$ copies of each degree constraint: For any solution $\chi_H \in STSP(n),$ we have that $\sum_{e \in \delta(v)} \left(\chi_H\right)_e = 2$ for any $v\in V$ by the degree constraints.  Every edge $e$ is incident to exactly $2$ vertices, so by adding the degree constraints over all $v\in V$, we obtain $$2\sum_{e\in E} \left(\chi_H\right)_e = 2n.$$  Multiplying this by $\f{1}{2}(d-2)$ yields that $$(d-2)\sum_{e\in E} \left(\chi_H\right)_e = n(d-2).$$ Since this equality is satisfied by every $\chi_H \in STSP(n),$ we can add it to Inequality (\ref{eq:main}).  Doing so yields Inequality (\ref{eq:TTF}): we add $d-2$ to the coefficient of every edge, and add $n(d-2)$ to the right hand side: $$n-2+n(d-2)=n-2+\f{n}{2}(n-4) = \f{2n-4+n(n-4)}{2}=\f{n^2-2n-4}{2}=\f{n^2}{2}-n-2.$$  Hence,  Inequality (\ref{eq:TTF}) remains valid and facet defining.  It remains to show that  Inequality (\ref{eq:TTF}) is in tight triangular form.

\begin{lm}
Inequality (\ref{eq:TTF}) is in tight triangular form.
\end{lm}

\begin{proof}
First, we argue that  $f_{ij} + f_{jk} \geq f_{ik}$ for all distinct triples $i, j, k$.  Note that for any edge $e$, $d-2\leq f_e \leq d-2+(d-1).$  If either $f_{ij} \geq d-1$ or $f_{jk}\geq d-1$ (or both) we have that $$f_{ik} \leq d-2 + (d-1) \leq f_{ij}+f_{jk}.$$  Hence the $f_{ij} + f_{jk} \geq f_{ik}$, except possibly in the case where $f_{ij}=f_{jk}=d-2.$  This case, however, requires that both $\{i, j\}$ and $\{j, k\}$ be edges of length $d=n/2$, so that $i$ and $k$ are not distinct.  Thus   $f_{ij} + f_{jk} \geq f_{ik}.$

We must also show that, for each $j\in V$, there exist some $i, k \in V\backslash\{j\}$ such that $f_{i, j}+f_{j, k}=f_{i, k}.$  By circulant symmetry, if this holds for some $j\in V$, it holds for all $j\in V$.  Without loss of generality, take $j=d+1.$  Then we take $i=1$ and $k=d$, so that $\{1, d\}$ is an edge of length $d-1$:
$$f_{i, k} = f_{1, d} = d-2+(d-1) = f_{1, d+1} + f_{d+1, d}=f_{i, j}+f_{j, k}.$$ 
Thus Inequality (\ref{eq:TTF}) is in tight triangular form.\hfill
\end{proof}

We can now readily compute the strength of our inequality.  
\begin{thm}
The strength of  Inequality (\ref{eq:TTF}) is $$\f{n^2-2n-4}{n^2-3n}\leq \f{11}{10}.$$  It is equal to $\f{11}{10}$ when $n=8.$
\end{thm}

\begin{proof}
By Theorem 2.11 in Goemans \cite{Goe95b}, its strength relative to the subtour elimination polyhedron $SP(n)$ is $$\f{\f{n^2}{2}-n-2}{\min\{fx:x\in SP(n)\}}.$$  Theorems 3.1 and 4.1 in Gutekunst and Williamson  \cite{Gut19b} show that $\min\{fx:x\in SP(n)\}$ is attained by $$x_e = \begin{cases} \f{1}{2}, & \ell_e = 1 \\ 1,& \ell_e =d \\ 0, & \text{ else}.\end{cases}$$  In the notation of Inequality (\ref{eq:TTF}), we find that $f_1=d-2+1=d-1$ while $f_d=d-2.$  The above solution places a total weight $d$ across all edges of length $d$, a total weight of $d$ across all edges of length $1$, and a total weight of 0 on all other edges.  Hence,
\begin{align*}
\min\{fx:x\in SP(n)\} &= d(d-1)+d(d-2) \\
&= d(2d-3) \\
&= \f{n}{2}(n-3).
\end{align*}
Thus the strength of the circlet inequality is
$$\f{\f{n^2}{2}-n-2}{\min\{fx:x\in SP(n)\}} = \f{\f{n^2}{2}-n-2}{\f{n^2}{2}-\f{3n}{2}}= \f{n^2-2n-4}{n^2-3n}.$$  \hfill
\end{proof}

We note that the circlet inequality appears to be marginally stronger than that of the crown inequality of  Naddef and Rinaldi \cite{Nad92}, which was also motivated by the subtour LP solution placing $1/2$ weight on every length-1 edge and $1$ weight on every length-$d$ edge.  The crown inequality has strength $$\f{3d(\f{d}{2}-1)-1}{3d(\f{d}{2}-1)-\f{d}{2}} = \f{3n(n-4)-8}{3n(n-4)-2n} =\f{3n^2-12n-8}{3n^2-14n}\leq \f{11}{10}.$$  The bound of $\f{11}{10}$ is also attained when $n=8$ by the crown inequality; otherwise ours is marginally stronger.

\section{Conclusions}

The main results of this paper introduce a new facet-defining inequality for the TSP.  This inequality is comparable to the crown inequalities Naddef and Rinaldi \cite{Nad92} in terms of the standard definition of strength.  We note, however, that this standard notation of strength is not as applicable to Circulant TSP.  If edge costs are metric, then a minimum-cost Hamiltonian cycle costs exactly the same as a minimum-cost Eulerian sub(multi)graph: Since a Hamiltonian cycle is itself Eulerian, 
$$\text{Cost of cheapest Eulerian sub(multi)graph } \leq \text{Cost of minimum-cost Hamiltonian cycle}.$$
When edge costs are metric, any Eulerian sub(multigraph) can, however, also be shortcut (see  Section 2.4 of Williamson and Shmoys \cite{DDBook} for details of shortcutting) to obtain a Hamiltonian cycle of no greater cost.  Hence:
$$\text{Cost of cheapest Eulerian sub(multi)graph } \geq \text{Cost of minimum-cost Hamiltonian cycle}.$$
On typical metric instances, then, optimizing over $STSP(n)$ and optimizing over $GTSP(n)$ yields the same solution.  That is not the case with circulant instances.

For example, consider the Circulant TSP instances motivating Inequality (\ref{eq:main}): instances where $4|n,$ edges of length 1 cost 1, and edges of length $d$ cost 0.  As argued in Gutekunst and Williamson \cite{Gut19b}, the optimal TSP solution (i.e., the minimum cost Hamiltonian cycle) uses 2 length-$d$ edges and hence costs $n-2.$  In contrast, Figure \ref{fig:euler} shows a minimum-cost Eulerian sub(multi)graph of cost $d=n/2.$  Hence, on this circulant instance, optimizing over $GTSP(n)$ instead of $STSP(n)$ is nearly a factor of 2 off. 

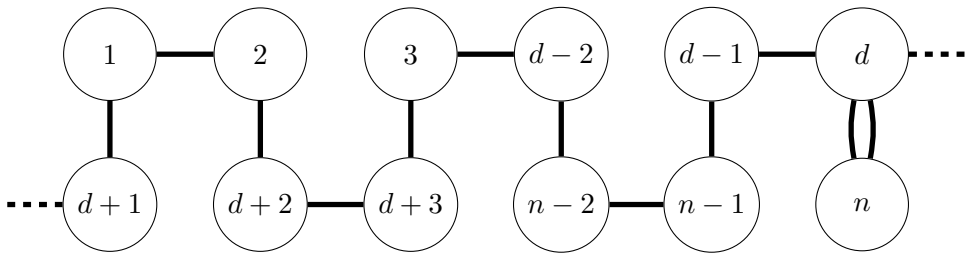
\begin{figure}[h!]
	\begin{center}

\begin{tikzpicture}[scale=0.5]
\tikzset{vertex/.style = {shape=circle,draw,minimum size=3.2em}}
\tikzset{edge/.style = {->,> = latex'}}
\tikzstyle{decision} = [diamond, draw, text badly centered, inner sep=3pt]
\tikzstyle{sq} = [regular polygon,regular polygon sides=4, draw, text badly centered, inner sep=3pt]
\node[vertex] (a) at  (0, 2) {$1$};
\node[vertex] (b) at  (4, 2) {$2$};
\node[vertex] (c) at  (8, 2) {$3$};
\node[vertex] (d) at  (12, 2) {$d-2$};
\node[vertex] (e) at  (16, 2) {$d-1$};
\node[vertex] (f) at  (20, 2) {$d$};
\node (h) at  (23, 2) {};

\node (h1) at  (-3, -2) {};
\node[vertex] (a1) at  (0, -2) {$d+1$};
\node[vertex] (b1) at  (4, -2) {$d+2$};
\node[vertex] (c1) at  (8, -2) {$d+3$};
\node[vertex] (d1) at  (12, -2) {$n-2$};
\node[vertex] (e1) at  (16, -2) {$n-1$};
\node[vertex] (f1) at  (20, -2) {$n$};

\draw[line width=2pt]  (a) -- (b);
\draw[line width=2pt]  (d) -- (c);
\draw[line width=2pt]  (e) -- (f);

\draw[line width=2pt]  (b1) -- (c1);
\draw[line width=2pt]  (d1) -- (e1);

\draw[line width=2pt]  (a) -- (a1);
\draw[line width=2pt]  (b) -- (b1);
\draw[line width=2pt]  (c) -- (c1);
\draw[line width=2pt]  (d) -- (d1);
\draw[line width=2pt]  (e) -- (e1);

\draw[line width=2pt]  (f) to [out=-100, in=100] (f1);
\draw[line width=2pt]  (f) to [out=-80, in=80] (f1);

\draw[line width=2pt, dashed]  (f) -- (h);
\draw[line width=2pt, dashed]  (h1) -- (a1);
\end{tikzpicture}
	\end{center}
	\caption{A minimum-cost Eulerian sub(multi)graph when length-1 edges cost 1, length-$d$ edges cost 0, and all other edge costs are arbitrarily large.  The dashed edge ``wraps around,'' connecting $d$ to $d+1$.}\label{fig:euler}\end{figure}

It is this discrepancy between $STSP(n)$ and $GTSP(n)$ on circulant instances that lends to the circlet inequality's weak strength: Circulant TSP is primarily concerned with understanding what combinations of edge lengths lead to a Hamiltonian tour, and shortcutting an Eulerian sub(multi)graph fundamentally changes edge lengths.  

We end this paper by asking three questions. First, given a solution to the subtour LP, can we efficiently determine whether or not it violates a circlet inequality (and if so, the labelling of the nodes that gives rise to the violated inequality)?  Second, what is the right analogue of strength for Circulant TSP inequalities?  Finally, the crown inequalities are also circulant.  What other circulant facet-defining inequalities are there for the TSP?  Such inequalities may help define the edge-length polytope of the TSP:
$$EL(n):= \text{conv}\{(t_1, ..., t_d): x\in STSP(n), t_i = \sum_{\{s, t\}\in E: \ell_{s, t} = i} x_{s, t}\}.$$  
 We hope that further polyhedral results on $EL(n)$ might open the door to new TSP results more generally, and bridge connections between combinatorial optimization and number theory.

 \section*{Acknowledgments}
This material is also based upon work supported by the National Science Foundation Graduate 
Research Fellowship Program under Grant No. DGE-1650441, and by National Science Foundation Grant No. CCF-1908517.  Any opinions,
findings, and conclusions or recommendations expressed in this material are those of the
authors and do not necessarily reflect the views of the National Science Foundation.

\bibliography{bibliog} 
\bibliographystyle{abbrv}

\appendix

\section{Casework for Propositions \ref{prop:cont1} and  \ref{prop:cont2}}

Throughout the casework, we refer to the groups indicated in the block picture.
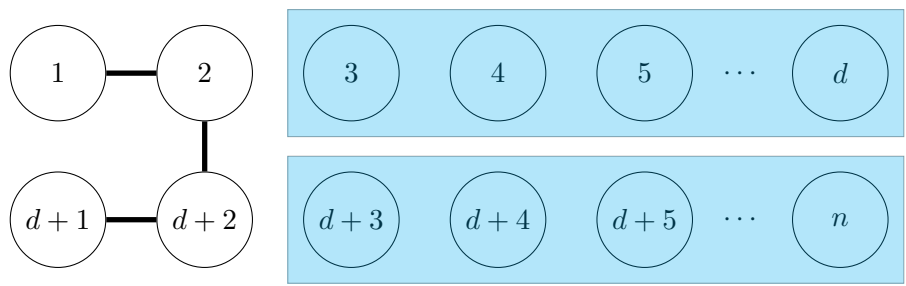
\begin{center}

\begin{tikzpicture}[scale=0.65]
\tikzset{vertex/.style = {shape=circle,draw,minimum size=3.3em}}
\tikzset{edge/.style = {->,> = latex'}}
\tikzstyle{decision} = [diamond, draw, text badly centered, inner sep=3pt]
\tikzstyle{sq} = [regular polygon,regular polygon sides=4, draw, text badly centered, inner sep=3pt]
\node[vertex] (a) at  (0, 2) {$1$};
\node[vertex] (b) at  (3, 2) {$2$};
\node[vertex] (c) at  (6, 2) {$3$};
\node[vertex] (d) at  (9, 2) {$4$};
\node[vertex] (e) at  (12, 2) {$5$};
\node (f) at  (14, 2) {$\cdots$};
\node[vertex] (g) at  (16, 2) {$d$};

\node[vertex] (a1) at  (0, -1) {$d+1$};
\node[vertex] (b1) at  (3, -1) {$d+2$};
\node[vertex] (c1) at  (6, -1) {$d+3$};
\node[vertex] (d1) at  (9, -1) {$d+4$};
\node[vertex] (e1) at  (12, -1) {$d+5$};
\node (f1) at  (14, -1) {$\cdots$};
\node[vertex] (g1) at  (16, -1) {$n$};

\draw[line width=2pt]  (a) to (b);
\draw[line width=2pt]  (b1) to (b);
\draw[line width=2pt]  (a1) to (b1);

\draw[fill=cyan, opacity=0.3] ($(c)+(-1.3, 1.3)$)  rectangle ($(g)-(-1.3, 1.3)$);

\draw[fill=cyan, opacity=0.3] ($(c1)+(-1.3, 1.3)$)  rectangle ($(g1)-(-1.3, 1.3)$);
\end{tikzpicture}
\end{center}
The {\bf top group} consists of vertices $s$ where $3\leq s\leq d,$ while the {\bf bottom group} consists of vertices $s$ where $d+3\leq s\leq n.$

\subsection*{Casework for Proposition \ref{prop:cont1}}

We first consider casework for Proposition \ref{prop:cont1}.  We want to show that $$c_{j, 1}  + c_{d+1, k} - c'_{j', k'} \geq 2,$$ and so we show  $$c_{j, 1}  + c_{d+1, k} - c'_{j', k'} - 2\geq 0$$  To do so, we consider the possible parities of $j$ and $k$ which determine, e.g., if $c_{j, 1}=\ell_{j, 1}$ or $c_{j, 1}=d-\ell_{j, 1}.$

\subsubsection*{Case 1: $j$ and $k$ are both even}
In this case,
\begin{equation}\label{eq:1a}c_{j, 1}  + c_{d+1, k} - c'_{j', k'} - 2 = \ell_{1, j} + \ell_{d+1, k} - (d-2-\ell'_{j', k'})-2.\end{equation}
\begin{itemize}
\item If $j, k$ are in the same group, then $\ell'_{j', k'} = \ell_{j, k}=|j-k|.$  Without loss of generality, suppose that $j, k$ are in the top group.  Then Equation (\ref{eq:1a}) becomes  $$c_{j, 1}  + c_{d+1, k} - c'_{j', k'} - 2 =(j-1)+(d+1-k)-(d-2-|j-k|)-2=j-k+|j-k|\geq 0,$$ since $x+|x|\geq 0$.  (If they were in bottom group, it would be $k-j+ |k-j|\geq 0$.)
\item If $j$ is on top and $k$ is on bottom, then $$\ell'_{j', k'} = \ell_{j, k}-2 = \begin{cases} (k-j)-2, & k<d+j \\ n-(k-j)-2, & \text{ else}. \end{cases}$$  Then Equation (\ref{eq:1a}) becomes  
\begin{align*}
c_{j, 1}  + c_{d+1, k} - c'_{j', k'} - 2&= (j-1) + (k-(d+1)) - \left(d-2-\ell'_{j', k'}\right) - 2\\
&= j+k -n-2 + \ell'_{j', k'} \\
&= j+k-n-2 + \begin{cases} (k-j)-2, & k<d+j \\ n-(k-j)-2, & \text{ else} \end{cases}\\
&= \begin{cases} 2k-n-4, & k<d+j \\ 2j-4, & \text{ else.} \end{cases}
\end{align*}
Both are nonnegative as $j\geq 3$ and $k\geq d+3.$

\item If $k$ is on top and $j$ is on bottom, then $$\ell'_{j', k'} = \ell_{j, k}-2 = \begin{cases} (j-k)-2, & j<d+k \\ n-(j-k)-2, & \text{ else}. \end{cases}$$  Then Equation (\ref{eq:1a}) becomes  
\begin{align*}
c_{j, 1}  + c_{d+1, k} - c'_{j', k'} - 2&= (n-j+1) + (d+1-k) - \left(d-2-\ell'_{j', k'}\right) - 2\\
&= n-j-k+2 +\ell'_{j', k'} \\
&= n-j-k+2 + \begin{cases} (j-k)-2, & j<d+k \\ n-(j-k)-2, & \text{ else} \end{cases}\\
&= \begin{cases} n-2k, & j<d+k \\  2n-2j, & \text{ else.} \end{cases}
\end{align*}
Both are nonnegative as $k\leq d$ and $j\leq n.$
\end{itemize}

\subsubsection*{Case 2: $j$ and $k$ are both odd}
In this case,
\begin{equation}\label{eq:1b}c_{j, 1}  + c_{d+1, k} - c'_{j', k'} - 2 = d-\ell_{1, j} +d- \ell_{d+1, k} - (d-2-\ell'_{j', k'})-2 = d-\ell_{1, j}- \ell_{d+1, k}+\ell'_{j', k'}.\end{equation}
\begin{itemize}
\item If $j, k$ are in the same group, then $\ell'_{j', k'} = \ell_{j, k}=|j-k|.$  Without loss of generality, suppose that $j, k$ are in the top group.  Then Equation (\ref{eq:1b}) becomes  $$d-\ell_{1, j}- \ell_{d+1, k}+\ell'_{j', k'} = d-(j-1)-(d+1-k)+|j-k|=k-j+|k-j|\geq 0,$$ since $x+|x|\geq 0$.  (If they were in bottom group, it would be $j-k+ |j-k|\geq 0$.)
\item If $j$ is on top and $k$ is on bottom, then $$\ell'_{j', k'} = \ell_{j, k}-2 = \begin{cases} (k-j)-2, & k<d+j \\ n-(k-j)-2, & \text{ else}. \end{cases}$$   Equation (\ref{eq:1b}) becomes  
\begin{align*}
d-\ell_{1, j}- \ell_{d+1, k}+\ell'_{j', k'}&= d - (j-1) - (k-(d+1)) +\ell'_{j', k'}\\
&= n+2-j-k +\ell'_{j', k'}\\
&= n+2-j-k +\begin{cases} (k-j)-2, & k<d+j \\ n-(k-j)-2, & \text{ else} \end{cases} \\
&= \begin{cases} n-2j, & k<d+j \\ 2n-2k, & \text{ else} \end{cases}
\end{align*}
Both are nonnegative as $j\leq d$ and $k\leq n.$
\item If $k$ is on top and $j$ is on bottom, then $$\ell'_{j', k'} = \ell_{j, k}-2 = \begin{cases} (j-k)-2, & j<d+k \\ n-(j-k)-2, & \text{ else}. \end{cases}$$   Equation (\ref{eq:1b}) becomes  
\begin{align*}
d-\ell_{1, j}- \ell_{d+1, k}+\ell'_{j', k'}&= d - (n-j+1) - (d+1-k) +\ell'_{j', k'}\\
&=  j+k-n-2 +\ell'_{j', k'}\\
&=   j+k-n-2 + \begin{cases} (j-k)-2, & j<d+k \\ n-(j-k)-2, & \text{ else} \end{cases}\\
&=    \begin{cases}  2j - n -4, & j<d+k \\   2k-4, & \text{ else} \end{cases}\\
\end{align*}
Both are nonnegative as $j\geq d+3$ and $k\geq 3.$
\end{itemize}

\subsubsection*{Case 3: $j$ and $k$ have opposite parity}
Without loss of generality, we let $j$ be odd and $k$ be even.  Then
\begin{equation}\label{eq:1c}c_{j, 1}  + c_{d+1, k} - c'_{j', k'} - 2 = d-\ell_{1, j} + \ell_{d+1, k} -\ell'_{j', k'}-2. \end{equation}
\begin{itemize}
\item If $j, k$ are both in the top group, then $\ell'_{j', k'} = \ell_{j, k}=|j-k|.$ 
 Then Equation (\ref{eq:1c}) becomes  $$d-\ell_{1, j} + \ell_{d+1, k} -\ell'_{j', k'}-2 = d-(j-1)+(d+1-k) - |j-k| - 2 = n-j-k-|j-k| = \begin{cases} n-2j, & j>k \\ n-2k,& \text{else} \end{cases} \geq 0,$$ since $j, k \leq d$.  
\item If $j, k$ are both in the bottom group, then $\ell'_{j', k'} = \ell_{j, k}=|j-k|.$ 
 Equation (\ref{eq:1c}) becomes  $$d-\ell_{1, j} + \ell_{d+1, k} -\ell'_{j', k'}-2 =d-(n-j+1)+ (k-(d+1))- |j-k| - 2 = j+k- |j-k|-n - 4= \begin{cases} 2k-n-4, & j>k \\ 2j-n-4,& \text{else,} \end{cases}$$ which is nonnegative since $j, k \geq d+3$.  
\item If $j$ is on top and $k$ is on bottom, then $$\ell'_{j', k'} = \ell_{j, k}-2 = \begin{cases} (k-j)-2, & k<d+j \\ n-(k-j)-2, & \text{ else}. \end{cases}$$   Equation (\ref{eq:1c}) becomes  
\begin{align*}
d-\ell_{1, j} + \ell_{d+1, k} -\ell'_{j', k'}-2 &= d-(j-1)+(k-(d+1))-\ell'_{j', k'}-2 \\
 &= k-j -2 -\ell'_{j', k'} \\
 &= k-j -2 -\begin{cases} (k-j)-2, & k<d+j \\ n-(k-j)-2, & \text{ else}\end{cases} \\
 &= \begin{cases} 0 , & k<d+j \\ 2k-2j  -n& \text{ else}.\end{cases} 
\end{align*}
Both are nonnegative, as in the second case $2k-2j\geq 2d=n.$
\item If $k$ is on top and $j$ is on bottom, then $$\ell'_{j', k'} = \ell_{j, k}-2 = \begin{cases} (j-k)-2, & j<d+k \\ n-(j-k)-2, & \text{ else}. \end{cases}$$   Equation (\ref{eq:1c}) becomes  
\begin{align*}
d-\ell_{1, j} + \ell_{d+1, k} -\ell'_{j', k'}-2 &=d-(n-j+1) + (d+1-k) -\ell'_{j', k'}-2 \\
 &= j-k-2 -\ell'_{j', k'} \\
 &= j-k-2 - \begin{cases} (j-k)-2, & j<d+k \\ n-(j-k)-2, & \text{ else}\end{cases} \\
 &= \begin{cases} 0, & j<d+k \\ 2j-2k-n, & \text{ else}\end{cases} 
\end{align*}
Both are nonnegative, as in the second case $2j-2k\geq 2d=n.$
\end{itemize}  

Up to symmetry, this shows that, in all possible cases, $$c_{j, 1}  + c_{d+1, k} - c'_{j', k'} - 2\geq 0.$$  For completeness, we include formulas for when $j$ is even and $k$ is odd.  These are:
\begin{itemize}
\item If $j, k$ are in the top, $j+k-|j-k|-4.$
\item If $j, k$ are in the bottom, $2n-j-k-|j-k|.$
\item If $j$ is on the top and $k$ is on the bottom, $$ \begin{cases} 0, & k > d+j \\ n+2j-2k, & \text{ else.}\end{cases}$$
\item If $j$ is on the bottom and $k$ is on the top, $$ \begin{cases} 0, & j > d+k \\ n-2j+2k, & \text{ else.}\end{cases}$$
\end{itemize}

\subsection*{Casework for Proposition \ref{prop:cont2}}

The computations for Proposition \ref{prop:cont2} follow as above.  We want to show that $$c_{j, d+2} + c_{d+1, k} - c'_{j', k'}-1 \geq 0,$$ and compute the left-hand side. Here we just write the results, highlighting those where the inequality can be tight.

\subsubsection*{Case 1: $j$ and $k$ are both even}
In this case,  \begin{equation}\label{eq:2a}c_{j, d+2} + c_{d+1, k} - c'_{j', k'}-1 = d-\ell_{j, d+2}+\ell_{d+1, k}-(d-2-\ell'_{j', k'})-1.\end{equation}
\begin{itemize}
\item If $j, k$ are in the top,  $\ell'_{j', k'} = \ell_{j, k}=|j-k|.$ Then Equation (\ref{eq:2a}) becomes:
$$d-\ell_{j, d+2}+\ell_{d+1, k}-(d-2-\ell'_{j', k'})-1 = d-(d+2-j)+(d+1-k)-(d-2-|j-k|)-1=j-k+|j-k|.$$
 {\bf This is zero when $j<k$.}
\item If $j, k$ are in the bottom, then Equation (\ref{eq:2a}) becomes:
$$d-\ell_{j, d+2}+\ell_{d+1, k}-(d-2-\ell'_{j', k'})-1= d-(j-(d+2))+(k-(d+1))-(d-2-|j-k|)-1=2 + |j-k| - (j-k).$$ This is never zero, since $|x|-x \geq 0$.
\item If $j$ is on the top and $k$ is on the bottom, then $$\ell'_{j', k'} = \ell_{j, k}-2 = \begin{cases} (k-j)-2, & k<d+j \\ n-(k-j)-2, & \text{ else}. \end{cases}$$  Equation (\ref{eq:2a}) becomes:
\begin{align*}
d-\ell_{j, d+2}+\ell_{d+1, k}-(d-2-\ell'_{j', k'})-1 &= d-(d+2-j)+(k-(d+1))-(d-2-\ell'_{j', k'})-1 \\
&=  j+k-2-n+\ell'_{j', k'} \\
&= \begin{cases}  j+k-n+ (k-j)-4, & k<d+j \\  j+k-n+n-(k-j)-4, & \text{ else} \end{cases}\\
&=\begin{cases}  2k-n-4, & k<d+j \\ 2j-4, & \text{ else} \end{cases}\\
\end{align*}
  Since $j\geq 3$ and $k\geq d+3,$ these are never zero.
\item If $j$ is on the bottom and $k$ is on the top, then $$\ell'_{j', k'} = \ell_{j, k}-2 = \begin{cases} (j-k)-2, & j<d+k \\ n-(j-k)-2, & \text{ else}. \end{cases}$$  Equation (\ref{eq:2a}) becomes:
\begin{align*}
d-\ell_{j, d+2}+\ell_{d+1, k}-(d-2-\ell'_{j', k'})-1 &= d-(j-(d+2))+(d+1-k)-(d-2-\ell'_{j', k'})-1 \\
&=n -j -k +4  +\ell'_{j', k'} \\
&=  \begin{cases}n  -2k +2, & j<d+k \\ 2n -2j  +2, & \text{ else}. \end{cases}
\end{align*}
Since $j \leq n$ and $k\leq d,$ these are never zero.
\end{itemize}

\subsubsection*{Case 2: $j$ and $k$ are both odd}
In this case,  \begin{equation}\label{eq:2b}c_{j, d+2} + c_{d+1, k} - c'_{j', k'}-1 = \ell_{j, d+2}+(d-\ell_{d+1, k})-(d-2-\ell'_{j', k'})-1 = \ell_{j, d+2}-\ell_{d+1, k}+\ell'_{j', k'}+1
\end{equation}
\begin{itemize}
\item If $j, k$ are in the top, $\ell'_{j', k'} = \ell_{j, k}=|j-k|.$ Then Equation (\ref{eq:2b}) becomes:
$$ \ell_{j, d+2}-\ell_{d+1, k}+\ell'_{j', k'}+1 =(d+2-j)-(d+1-k)+|j-k|+1 =2+|k-j|+(k-j).$$
This is never zero, since $|x|+x\geq 0.$
\item If $j, k$ are in the bottom,   $\ell'_{j', k'} = \ell_{j, k}=|j-k|.$   Then Equation (\ref{eq:2b}) becomes:
$$ \ell_{j, d+2}-\ell_{d+1, k}+\ell'_{j', k'}+1 = (j-(d+2))-(k-(d+1))+|k-j|+1  =|j-k|+j-k.$$ {\bf This is zero when $j<k$.}
\item If $j$ is on the top and $k$ is on the bottom, $$\ell'_{j', k'} = \ell_{j, k}-2 = \begin{cases} (k-j)-2, & k<d+j \\ n-(k-j)-2, & \text{ else}. \end{cases}$$  Then Equation (\ref{eq:2b}) becomes:
\begin{align*}
 \ell_{j, d+2}-\ell_{d+1, k}+\ell'_{j', k'}+1 &= (d+2-j)-(k-(d+1))+\ell'_{j', k'}+1 \\
&=n+4-j-k+\ell'_{j', k'} \\
&=n+4-j-k+\begin{cases} (k-j)-2, & k<d+j \\ n-(k-j)-2, & \text{ else} \end{cases} \\
&= \begin{cases}n+2-2j, & k<d+j \\ 2n+2-2k, & \text{ else}. \end{cases} \\
\end{align*}
Since $k \leq n$ and $j \leq d,$ these are never zero.
\item If $j$ is on the bottom and $k$ is on the top, $$\ell'_{j', k'} = \ell_{j, k}-2 = \begin{cases} (j-k)-2, & j<d+k \\ n-(j-k)-2, & \text{ else}. \end{cases}$$  
Then Equation (\ref{eq:2b}) becomes:
\begin{align*}
 \ell_{j, d+2}-\ell_{d+1, k}+\ell'_{j', k'}+1 &= (j-(d+2))-(d+1-k)+\ell'_{j', k'}+1 \\
&= j + k -n -2 +\ell'_{j', k'} \\
&= j + k -n -2 +  \begin{cases} (j-k)-2, & j<d+k \\ n-(j-k)-2, & \text{ else} \end{cases}  \\
&=  \begin{cases}  2j  -n -4, & j<d+k \\2k - 4, & \text{ else.} \end{cases}  \\
\end{align*}
  Since $k\geq 3$ and $j\geq d+3,$ these are never zero.
\end{itemize}

\subsubsection*{Case 3: $j$ is odd and  $k$ is even}
In this case,  \begin{equation}\label{eq:2c}
c_{j, d+2} + c_{d+1, k} - c'_{j', k'}-1 
= \ell_{j, d+2}+\ell_{d+1, k}-\ell'_{j', k'}-1
\end{equation}

\begin{itemize}
\item If $j, k$ are in the top, then $\ell'_{j', k'} = \ell_{j, k}=|j-k|.$  Equation (\ref{eq:2c}) becomes:
$$\ell_{j, d+2}+\ell_{d+1, k}-\ell'_{j', k'}-1
= (d+2-j) + (d+1-k) - |j-k| - 1  =n+2-j-k-|j-k| = \begin{cases} n+2-2j, & j > k \\n+2-2k, & \text{ else.}\end{cases}
$$  These are  never zero, since $j, k\leq d.$

\item If $j, k$ are in the bottom, then $\ell'_{j', k'} = \ell_{j, k}=|j-k|.$ Equation (\ref{eq:2c}) becomes:
$$ \ell_{j, d+2}+\ell_{d+1, k}-\ell'_{j', k'}-1
= j-(d+2) + k-(d+1) -|j-k|-1
=j+k-|j-k|-n-4 =  \begin{cases} 2k-n-4, & j > k \\2j-n-4, & \text{ else.}\end{cases} $$ These are  never zero, since $j, k\geq d+3.$
\item If $j$ is on the top and $k$ is on the bottom, then $$\ell'_{j', k'} = \ell_{j, k}-2 = \begin{cases} (k-j)-2, & k<d+j \\ n-(k-j)-2, & \text{ else}. \end{cases}$$   Equation (\ref{eq:2c}) becomes:
\begin{align*}
\ell_{j, d+2}+\ell_{d+1, k}-\ell'_{j', k'}-1 &= (d+2-j)+(k-(d+1))- \ell'_{j', k'} -1 \\
 &= k-j- \ell'_{j', k'} \\
 &= k-j-  \begin{cases} (k-j)-2, & k<d+j \\ n-(k-j)-2, & \text{ else} \end{cases}\\
 &=\begin{cases}2 & k<d+j \\ 2k-2j-n+2 & \text{ else.} \end{cases}
\end{align*}
  Since the second case occurs only when $k\geq d+j \implies 2k-2j-n\geq 0,$ these are never zero.
\item If $j$ is on the bottom and $k$ is on the top, then $$\ell'_{j', k'} = \ell_{j, k}-2 = \begin{cases} (j-k)-2, & j<d+k \\ n-(j-k)-2, & \text{ else}.\end{cases} $$ 
Equation (\ref{eq:2c}) becomes:
\begin{align*}
\ell_{j, d+2}+\ell_{d+1, k}-\ell'_{j', k'}-1
 &= (j-(d+2))+(d+1-k)- \ell'_{j', k'} -1 \\
 &= j-k-2- \ell'_{j', k'}  \\
 &= j-k-2-  \begin{cases} (j-k)-2, & j<d+k \\ n-(j-k)-2, & \text{ else}\end{cases}  \\
 &= \begin{cases} 0, & j<d+k \\2j-2k-n, & \text{ else.} \end{cases}
\end{align*}
The second case is not zero whenever  $j>d+k.$  Hence {\bf the equation is zero when $j\leq d+k$.}
\end{itemize}

\subsubsection*{Case 4: $j$ is even and  $k$ is odd}
In this case,  \begin{equation}\label{eq:2d}
c_{j, d+2} + c_{d+1, k} - c'_{j', k'}-1 
=(d- \ell_{j, d+2})+(d-\ell_{d+1, k})-\ell'_{j', k'}-1 
= n- \ell_{j, d+2}-\ell_{d+1, k}-\ell'_{j', k'}-1.
\end{equation}

\begin{itemize}
\item If $j, k$ are in the top, then $\ell'_{j', k'} = \ell_{j, k}=|j-k|.$  Equation (\ref{eq:2d}) becomes:
$$n- \ell_{j, d+2}-\ell_{d+1, k}-\ell'_{j', k'}-1 
= n - (d+2-j)-(d+1-k)-|j-k|-1 
= j+k-|j-k|-4 = \begin{cases} 2k-4, & j > k \\2j-4, & \text{ else.}\end{cases}
$$  These are  never zero, since $j, k\geq 3.$
\item If $j, k$ are in the bottom,  then $\ell'_{j', k'} = \ell_{j, k}=|j-k|.$  Equation (\ref{eq:2d}) becomes:
$$n- \ell_{j, d+2}-\ell_{d+1, k}-\ell'_{j', k'}-1 
= n - (j-(d+2))-(k-(d+1))-|j-k|-1
=  2n+2-j-k-|j-k| $$
and
$$ 2n+2-j-k-|j-k| =\begin{cases}2n+2-2j, & j > k \\2n+2-2k, & \text{ else.}\end{cases}.$$ These are  never zero, since $j, k\leq n.$

\item If $j$ is on the top and $k$ is on the bottom, then $$\ell'_{j', k'} = \ell_{j, k}-2 = \begin{cases} (k-j)-2, & k<d+j \\ n-(k-j)-2, & \text{ else}. \end{cases}$$  Equation (\ref{eq:2d}) becomes:
\begin{align*}
n- \ell_{j, d+2}-\ell_{d+1, k}-\ell'_{j', k'}-1 
&= n - (d+2-j)-(k-(d+1))-\ell'_{j', k'}-1 \\
&= n +j-k-2  -\ell'_{j', k'} \\
&= n +j-k-2  - \begin{cases} (k-j)-2, & k<d+j \\ n-(k-j)-2, & \text{ else} \end{cases}\\
&=  \begin{cases}  n +2j-2k  , & k<d+j \\ 0 , & \text{ else} \end{cases}\\
\end{align*}
 The first case is not zero because $d+j-k>0.$  Hence {\bf the equation is zero when $k\geq d+j$.}

\item If $j$ is on the bottom and $k$ is on the top, then $$\ell'_{j', k'} = \ell_{j, k}-2 = \begin{cases} (j-k)-2, & j<d+k \\ n-(j-k)-2, & \text{ else}.\end{cases} $$ 
Equation (\ref{eq:2d}) becomes:
\begin{align*}
n- \ell_{j, d+2}-\ell_{d+1, k}-\ell'_{j', k'}-1 
&=n - (j-(d+2))-(d+1-k)-\ell'_{j', k'}-1 \\
&=n -j  +k-\ell'_{j', k'} \\
&=n -j  +k-  \begin{cases} (j-k)-2, & j<d+k \\ n-(j-k)-2, & \text{ else}\end{cases} \\
&=\begin{cases} n+2-2j+2k, & j<d+k \\2, & \text{ else.}\end{cases} \\
\end{align*}
 The first case is not zero because $j< d+k$ implies $0< n+2k-2j.$
\end{itemize}

In all cases, $$c_{j, d+2} + c_{d+1, k} - c'_{j', k'}-1 \geq 0,$$ with equality only when:
\begin{itemize}
\item $j, k\leq d$, both $j, k$ even, and $j<k.$
\item $j, k\geq d+3$, both $j, k$ odd, and $j<k.$
\item $j\geq d+3$ and odd, $k\leq d$ and even, and $j\leq d+k.$
\item $j\leq d$ and even, $k\geq d+3$ and odd, and $k\geq d+j.$
\end{itemize}

This means that equality can hold only in those cases mentioned in the proof of Proposition \ref{prop:cont2}.

 \end{document}